\documentclass[10pt,a4paper,leqno]{amsart}
\usepackage{amssymb}
\usepackage{graphicx}
\usepackage{pdfsync}
\usepackage{xspace}
\usepackage{fancyhdr}
\usepackage[english]{babel}
\usepackage{amsfonts}
\usepackage{latexsym}
\usepackage{graphics}
\usepackage{hyperref}
\usepackage{pifont,bbding}
\usepackage{color}
\usepackage{bm}
\usepackage{enumerate,amsthm,amsmath,comment,psfrag}
 \usepackage{mathrsfs} 
\newtheorem{remark}{Remark}[section]
\newtheorem{lemma}{Lemma}[section]
\newtheorem{proposition}{Proposition}[section]
\newtheorem{theorem}{Theorem}[section]

\numberwithin{equation}{section}
\newcommand{\ep}{\varepsilon}

\newcommand{\la}{\lambda}
\newcommand{\Om}{\varOmega}
\newcommand{\D}{\varDelta}
\newcommand{\E}{\mathcal{E}}
\newcommand{\ii}{\mathrm{i}}
\newcommand{\pa}{\bar{\partial}}
\newcommand{\pt}{\tilde{\partial}}
\newcommand{\Hf}{\mathcal{H}}
\newcommand{\p}{\partial}
\newcommand{\tha}{t_{n+\frac12}}
\DeclareMathOperator\Rea{\mathrm{Re}}

\newcommand{\V}{\mathbb{V}}
\newcommand{\h}{\eta_{2,\V^n}}
\newcommand{\hin}{\eta_{\infty,\V^n}}
\newcommand{\hhin}{\eta_{\infty,\V^{n+1}}}
\newcommand{\hh}{\eta_{2,\V^{n+1}}}
\newcommand{\het}{\eta_{2,\widehat\V^{n+1}}}
\newcommand{\R}{\mathcal{R}}
\newcommand{\M}{\mathcal{M}}
\newcommand{\N}{\mathcal{N}}

\newcommand{\Nb}{\mathbb{N}}

\newcommand{\I}{\mathcal{I}}
\newcommand{\II}{\mathrm{I}}
\newcommand{\T}{\mathcal{T}}
\newcommand{\Pro}{\mathcal{P}}
\newcommand{\TT}{\mathrm{T}}
\newcommand{\Ss}{\mathrm{S}}
\newcommand{\Les}{ \mathscr{L}}
\begin{document}
%
%
%
%
\title[A posteriori error analysis for evolution NLS]
{
A posteriori error analysis for evolution nonlinear Schr\"odinger equations up to the critical exponent}
\date{\today}
\author{Theodoros Katsaounis}
\address[Theodoros Katsaounis]{King Abdullah University of Science and Technology(KAUST)\\
Thuwal, Kingdom of Saudi Arabia \\
 \& IACM--FORTH, Heraklion, GREECE \\
 \& Dept. of Math. \& Applied Mathematics, Univ. of Crete, Greece
}
\email{theodoros.katsaounis@kaust.edu.sa}
\author{Irene Kyza}
\address[Irene Kyza]{Division of Mathematics\\
University of Dundee\\
Dundee DD1 4HN\\
Scotland, UK} \email{ikyza@maths.dundee.ac.uk}

\thanks{Work partially supported by Excellence Award 1456 of the Greek Ministry of Research and Education.}
\begin{abstract}
We provide a posteriori error estimates  in the $L^\infty(L^2)-$norm for relaxation  time discrete and fully discrete schemes for a class of evolution nonlinear Schr\"odinger equations up to the critical exponent. In particular for the discretization in time we use the relaxation Crank-Nicolson-type scheme introduced by Besse in \cite{Besse}. For the discretization in space we use finite element spaces that are allowed to change between time steps. The estimates are obtained using the reconstruction technique. Through this technique the problem is converted to a perturbation of the original partial differential equation and this makes it possible to use nonlinear stability arguments
as in the continuous problem. In particular, main ingredients we use in our analysis are the Gagliardo-Nirenberg inequality and the two conservation laws (mass and energy conservation) of the continuous problem. Numerical results illustrate that the estimates are indeed of optimal order of convergence. 
\end{abstract}
\maketitle
%
%
\section{Introduction}
%
In this paper we provide error control for a class of evolution nonlinear Schr\"odinger (NLS) equations, up to the critical exponent through \emph{rigorous a posteriori error analysis}. \emph{To the best of our knowledge, this is the first time that rigorous a posteriori error estimates are proven for evolution NLS equations.}
More specifically, we consider the initial and boundary value problem 
\begin{equation}
\label{NLS} \left \{
\begin{aligned}
&\p_t u-\mathrm{i}\alpha\varDelta u=\ii\la f(u)    &&\quad\mbox{in ${\varOmega}\times (0,T)$,}& \\
&u=0  &&\quad\mbox{on $\partial \varOmega\times (0,T]$,}&  \\
&u(\cdot,0)=u_0 &&\quad\mbox{in ${\varOmega}$},&
\end{aligned}
\right.
\end{equation}
where $\Om\subset\mathbb{R}^d$, $d=1,2,$ is a bounded, convex, polygonal domain for $d=2$ and a finite interval for $d=1,$ $T<\infty$, $\alpha>0$, $\la\in\mathbb{R},$ and where $f(u)$ denotes the nonlinear term. In particular, the nonlinear term has the form 
$$f(z):=|z|^{2p}z,\quad \frac12\le p\le p^*\ \text{ with }\ p^*:=\frac2d;$$ 
$p^*=\frac2d$ is called  the critical exponent in spatial dimension $d$. For $u_0\in H_0^1(\Om)\cap H^2(\Om)$ and $\frac12\le p< p^*$ it is well known that  \eqref{NLS}    admits a unique solution $u\in C\left([0,T];  H_0^1(\Om)\cap H^2(\Om)\right)\cap C^1\left([0,T]; L^2(\Om)\right)$,   cf., e.g., \cite{Anton, Bourgain, BG}.    In contrast, for the critical exponent $p^*=\frac2d$, global existence of a solution $u$ of \eqref{NLS} is guaranteed, provided that a smallness assumption for the initial condition $u_0$ holds; cf. equation \eqref{initialgrowth}.

Models of the form \eqref{NLS} are widely used in many areas of applied sciences. For example they appear in nonlinear optics and lasers \cite{HK}, water waves \cite{Dysthe}, quantum hydrodynamics \cite{JMR} and Bose-Einstein condensates \cite{Strecker}. More applications are discussed in \cite{Sulem}.

There is still a large activity on NLS equations in the area of partial differential equations (PDEs) and analysis community, cf. e.g., \cite{Anton, Bourgain, Cazenave, MR1, MR2, MR3, OT, Tao} and the references therein. 
Moreover this activity includes qualitative and asymptotic questions, cf. e.g., \cite{BT, Carles, LMill,TVZ} and the references therein. A particular example is the semiclassical behavior of NLS equations, i.e. the regime where $0<\alpha\ll 1,$ $\lambda\sim \frac1\alpha$. 

Those are the main reasons that \eqref{NLS} has attracted the interest of the numerical analysis community as well. Several papers exist in the literature dealing with  discretization methods for \eqref{NLS} and their stability and convergence properties through a priori error analysis; here we mention a few: \cite{ADK,AAHM,BC,BJM,Besse,DFP,KM1,KM2,Tha,Wang, Zouraris}. Popular methods for the discretization in time of \eqref{NLS} are Crank-Nicolson-type and time-splitting-type methods, while for the spatial discretization spectral or finite element methods are usually used.

However, there is a very limited literature on the a posteriori error control for the evolution Schr\"odinger equations. For linear evolution Schr\"odinger equations a posteriori error estimates for Crank-Nicolson finite element schemes can be found in \cite{D,KK,Kyza}, whilst for time-splitting spectral methods can be found in \cite{KMP}. To the best of our knowledge \emph{no a posteriori error bounds exist in the literature for evolution NLS equations}. Nevertheless, developing such estimates is important. Not only they will provide mathematical guarantees on how accurate the numerical approximation is, but they will also highlight qualitative characteristics of the exact solution of \eqref{NLS} not known before, via rigorous error control. Additionally, the a posteriori error bounds can lead to the development of an adaptive algorithm that will significantly reduce the computational cost. That was the case for example in \cite{KK} with the a posteriori estimator of linear evolution Schr\"odinger equations.  Adaptive algorithms based on \emph{ad hoc} mesh selection criteria exist in the literature for various cases of NLS equations \eqref{NLS}, cf. \cite{ADKM,JLMPV, P,TA,ZZZ}. Usually the criteria used for the construction of adaptive algorithms in these cases are based on structural properties known for the exact solution.

\emph{Our main contribution in this paper is the rigorous proof of optimal order a posteriori error bounds in the $L^\infty(L^2)-$norm for the NLS equation \eqref{NLS} up to the critical exponent}, when it is discretized by a relaxation Crank-Nicolson finite element scheme. With the term \textit{optimal order} we mean that the a posteriori  estimator reduces with the same order  as the exact error. The relaxation Crank-Nicolson scheme we use for the  discretization in time is a generalization to variable time steps of the relaxation scheme introduced earlier by Besse in \cite{Besse} for the time discretization of \eqref{NLS} for constant time steps. The reason we also use the relaxation scheme,  rather than the standard Crank-Nicolson scheme, is because that way the nonlinear term is computed explicitly. Thus we avoid solving a nonlinear equation, that would have added an error, difficult to handle a posteriori. Moreover, the relaxation scheme exhibits mass conservation, same as the standard Crank-Nicolson scheme, thus reflecting the mass conservation  of the continuous problem \eqref{NLS}, cf. \eqref{cl1} below.

The a posteriori error estimates will be obtained using the reconstruction technique proposed by Akrivis, Makridakis \& Nochetto, \cite{AMN1,MN1}. Through this technique we will be able to obtain an error equation of a similar form to the NLS equation in \eqref{NLS}. The derivation of the estimates is then based on energy techniques and on nonlinear stability arguments for \eqref{NLS}. Actually the PDE satisfied by the error is harder than the NLS equation itself. For this reason the handling of the nonlinearity is delicate and technical. 

We next mention some of the main tools  used in the subsequent analysis. For this, we need to introduce some notation. We denote by $\|\cdot\|$ the $L^2-$norm in $\Om$, while for $1\le q\le\infty$, $q\ne 2$, we denote by $\|\cdot\|_{L^q}$ the $L^q-$norm in $\Om$. We equip $H_0^1(\Om)$ with the norm $\|\nabla v\|$ and we denote by $H^{-1}(\Om)$ the dual of $H_0^1(\Om)$ under that norm; we denote by $\|\cdot\|_{H^{-1}}$ the norm in $H^{-1}(\Om)$.
By $\langle\cdot,\cdot\rangle$ we indicate both the $L^2-$inner product, or the $H^{-1}-H_0^1$ duality pairing in $\Om$, depending on the context. 
In what follows,  global constants or functionals depending on the initial condition $u_0$ that are introduced in the paper may also depend on the dimension $d$, the exponent $p$ and the parameters $\alpha$ and $\la$. For simplicity, where there is no confusion, we avoid writing those dependences on the definitions, but we mention them precisely each time we define such a quantity (e.g., \eqref{GN} and \eqref{initialgrowth} below).

A key role in the derivation of the a posteriori estimates of the paper will be played by the Gagliardo-Nirenberg inequality, \cite{Brezis}
\begin{equation}
\label{GN}
\|v\|_{L^{2p+2}}\le\beta\|\nabla v\|^{\zeta}\|v\|^{1-\zeta},\quad \forall v\in H_0^1(\Om),
\end{equation}
where $\zeta:=\displaystyle\frac{pd}{2(p+1)}$ and $\beta$ is an absolute constant depending on $\Om$, $d$ and $p$. For the cases of critical exponent  $p=p^*(=\frac2d)$, \eqref{NLS} admits a unique solution $u\in C\left([0,T];  H_0^1(\Om)\cap H^2(\Om)\right)\cap C^1\left([0,T]; L^2(\Om)\right)$, if $u_0\in H_0^1(\Om)\cap H^2(\Om)$  and
\begin{equation}
\label{initialgrowth}
\Gamma(u_0):=\frac{\beta^{2p+2}\la}{\alpha(p+1)}\|u_0\|^{2p}<1.
\end{equation}
 For $\lambda \le0$, inequality \eqref{initialgrowth} is automatically satisfied;  problem \eqref{NLS} is then known as the {defocusing NLS equation}. If $\lambda>0$,  problem \eqref{NLS} is called the {focusing NLS equation}. In this case, if \eqref{initialgrowth} is not satisfied, i.e., if $\Gamma(u_0)\ge1,$ the solution $u$ of \eqref{NLS} may blow up in the $H^1-$norm in some finite time $T^*<\infty$ (cf., e.g., \cite{MR1, MR2, MR3} and the references therein). The focusing cases we consider in this paper are those with $\Gamma(u_0)<1$. A posteriori error analysis for controlling the error close to blowup for focusing cases with $\Gamma(u_0)\ge 1$ is a very interesting question and currently under investigation.

Problem \eqref{NLS} satisfies  two conservation laws that will be instrumental in the subsequent analysis. In particular, for $t\ge 0$, we have 
\begin{align}
 \|u(t)\| & =\|u_0\|, \quad  & \text{\emph{mass conservation}}, \label{cl1} \\
 \|\nabla u(t)\|^2 - \frac{\la}{\alpha(p+1)}\|u(t)\|^{2p+2}_{L^{2p+2}} & = \|\nabla u_0\|^2  - \frac{\la}{\alpha(p+1)} \|u_0\|^{2p+2}_{L^{2p+2}}, \quad & \text{\emph{energy conservation}} . \label{cl2}
\end{align}

The final estimates include an exponential term of the $L^{2p}(L^\infty)-$norm of the approximation; this is due to Gronwall's inequality. This is an improvement compared to the existing results on the a priori error analysis, where the exponential of the $L^\infty(L^\infty)-$norm of the approximation appears. Although any exponential term may seem pessimistic, it actually reflects the nonlinear nature of the problem. A similar term appears in the  a posteriori error analysis for semilinear parabolic equations with possible blowup in finite time; cf. \cite{CGKM,KM}. In fact, as illustrated in \cite{CGKM}, this term enables the proposition of an efficient space-time adaptive algorithm leading to blowup detection and accurate numerical approximation of blowup times. In the cases of NLS equations \eqref{NLS} we expect that the exponential term will also be proven beneficial towards the development of an efficient time-space adaptive algorithm in the spirit of \cite{CGKM}. This is the subject of a forthcoming paper. In this paper we investigate the behavior of the exponential term numerically in the last section of the paper.

The paper is organized as follows. In Sections~\ref{timediscr},~\ref{apostimdis} we consider only time discretization using a relaxation Crank-Nicolson-type scheme. In particular, in Section \ref{timediscr} we introduce some additional notation and generalize the relaxation scheme of \cite{Besse} to variable time steps. We define an appropriate time reconstruction and study its properties. Section \ref{apostimdis} is devoted to the proof of optimal order  a posteriori error estimates in the  $L^{\infty}(L^2)-$norm for the time discrete scheme. The analysis requires a careful use of the Gagliardo-Nirenberg inequality; the two conservation laws are then used for the boundedness of the solution in the $H^1-$norm. The fully discrete relaxation Crank-Nicolson finite element  scheme is introduced in Section~\ref{fullydiscr}. In this case, we use finite element spaces that are allowed to change from one time step to another. With the help of the elliptic reconstruction of \cite{MN1} and the time reconstruction of Section~\ref{timediscr} we define an appropriate time-space reconstruction and study its properties. In Section~\ref{FDapost} we provide a posteriori error estimates for the fully discrete scheme. To obtain the estimates we use the machinery developed in Section~\ref{apostimdis} for the time discrete scheme as well as residual-type estimators to control the terms coming from the use of the elliptic reconstruction. The analysis is quite technical, but eventually leads to a posteriori error bounds in the $L^\infty(L^2)-$norm that are expected to be of optimal order of accuracy. Finally, in Section~\ref{numexp} we present numerical experiments, using uniform partitions in space and time. More specifically, we verify that the obtained estimators are indeed of optimal oder of accuracy and we study the  behavior of the exponential term.
%
\section{Time discrete schemes}\label{timediscr}
%
As already mentioned in the Introduction, we first consider time discrete schemes, in an attempt to present clearly the main ideas of the technical analysis caused by the nonlinear nature of problem \eqref{NLS}.
%
\subsection{A  relaxation Crank-Nicolson-type method}
We consider a partition $0=:t_0<t_1<\cdots<t_N:=T$ of $[0,T]$ and let $k_n:=t_{n+1}-t_n$ and $I_n:=(t_n,t_{n+1}]$, $0\le n\le N-1,$ denote the variable time steps and subintervals of $[0,T],$ respectively. Let also $k:=\max_{1\le n\le N}k_n$. We further assume that there exists an absolute constant $c\in\mathbb{R}^+$ such that
\begin{equation}
\label{mildstepcons}
k_n\le c k_{n-1},\quad 1\le n\le N-1.
\end{equation}
The mild constraint \eqref{mildstepcons} between consecutive time steps appears in the analysis of other time-stepping methods with variable time steps. For example, it appears in \cite{MN2}, in the a posteriori error analysis for discontinuous Galerkin in time methods.

The relaxation Crank-Nicolson-type scheme for \eqref{NLS} is defined as follows: We seek approximations $U^n\in H_0^1(\Om)$ to $u(t_n)$, $0\le n\le N,$ such that
\begin{equation}
\label{CNrelax} \left \{
\begin{aligned}
& \frac{k_{n-1}}{k_n+k_{n-1}}\Phi^{n+\frac12}+\frac{k_n}{k_n+k_{n-1}}\Phi^{n-\frac12}=|U^n|^{2p},\quad 0\le n\le N-1,    \\
&\pa U^n-\ii\alpha\D U^{n+\frac12}=\ii\lambda\Phi^{n+\frac12}U^{n+\frac12},\qquad\hspace{1.35cm} 0\le n\le N-1,
\end{aligned}
\right.
\end{equation}
with $k_{-1}:=k_0$, $\Phi^{-\frac12}=|u_0|^{2p}$ and $U^0=u_0$. In \eqref{CNrelax} we also used the notation
\begin{equation}
\label{notation}
\pa U^n:=\frac{U^{n+1}-U^n}{k_n}\ \text{ and }\ U^{n+\frac12}:=\frac{U^{n+1}+U^n}2.
\end{equation}
The relaxation scheme \eqref{CNrelax} was introduced for constant time steps by Besse in \cite{Besse} for the numerical solution of \eqref{NLS}. Here we generalise Besse's method to the case of variable time steps.

We next present briefly how method \eqref{CNrelax} can be obtained to make it clear to the reader that the generalisation of Besse's scheme in \cite{Besse} indeed leads to \eqref{CNrelax} for variable time steps. First we rewrite the NLS equation in \eqref{NLS} equivalently as the system of the following two equations:
\begin{equation}
\label{NLSsystem} \left \{
\begin{aligned} 
&\phi=|u|^{2p} &&\quad\mbox{in ${\varOmega}\times (0,T]$,}& \\
& \p_tu-\ii\alpha\D u=\ii\lambda\phi u &&\quad\mbox{in ${\varOmega}\times (0,T]$.}& 
\end{aligned}
\right.
\end{equation}

Recall that the Crank-Nicolson method for \eqref{NLS} reads as: for $0\le n\le N,$ find $U^n\in H_0^1(\Om)$ such that
\begin{equation}
\label{CN}
\pa U^n-\ii\alpha\D U^{n+\frac12}=\ii\lambda|U^{n+\frac12}|^{2p}U^{n+\frac12},\quad 0\le n\le N-1,
\end{equation}
with $U^0=u_0$. Note that using the equivalent system-form \eqref{NLSsystem}, $|U^{n+\frac12}|^{2p}$ is an approximation of
$$\phi(t_{n+\frac12})=|u(t_{n+\frac12})|^{2p}\ \text{ with }\ t_{n+\frac12}:=\frac{t_{n+1}+t_n}2.$$ 
The idea now is to replace $|U^{n+\frac12}|^{2p}$ in \eqref{CN} by an appropriate approximation of $\phi(\tha)$, in an attempt to avoid the costly numerical treatment of the nonlinearity.
This is where the first equation of \eqref{CNrelax} is involved. In particular, it is conjectured that if, $\Phi^{-\frac12}=|u_0|^{2p}$ and provided regularity on the exact solution $u$ of \eqref{NLS}, then
\begin{equation}
\label{approx1}
|u(\tha)|^{2p}=\phi(\tha)=\varPhi^{n-\frac12}+\frac{k_n+k_{n-1}}{k_{n-1}}\left(|u(t_n)|^{2p}-\Phi^{n-\frac12}\right)+\mathcal{O}(k^2),\ 0\le n\le N-1.
\end{equation}
Neglecting $\mathcal{O}(k^2)$ in \eqref{approx1} and replacing $\phi(t_{n+\frac12})$ by $\Phi^{n+\frac12}$ and $u(t_n)$ by $U^n$, we obtain the first equation in \eqref{CNrelax}. We prove conjecture \eqref{approx1} formally:
\begin{equation}
\label{approx2}
\phi(\tha)=\phi(t_{n-\frac12})+\frac{k_n+k_{n-1}}2\p_t\phi(t)|_{t=t_{n-\frac12}}+\mathcal{O}\left(\frac12(k_n^2+k_{n-1}^2)\right),
\end{equation}
where
\begin{equation}
\label{timeder}
\p_t\phi(t)|_{t=t_{n-\frac12}}=\frac2{k_{n-1}}\left(\phi(t_n)-\phi(t_{n-\frac12})\right)+\mathcal{O}(k_{n-1}).
\end{equation}
Invoking \eqref{timeder} in \eqref{approx2} we deduce
$$\phi(\tha)=\phi(t_{n-\frac12})+\frac{k_n+k_{n-1}}{k_{n-1}}\left(\phi(t_n)-\phi(t_{n-\frac12})\right)+\mathcal{O}(k^2),$$
since $k=\max_{0\le n\le N-1}k_n$. An inductive argument and \eqref{mildstepcons}  lead to \eqref{approx1} (recall that $\phi(t_n)=|u(t_n)|^{2p}$).

Note that with scheme \eqref{CNrelax}, the nonlinear term at each time $t_n$ is computed explicitly, avoiding the solution of a nonlinear equation, as in the case of the Crank-Nicolson method \eqref{CN}. At the same time, the term involving the laplacian ($\D u$) is discretized in time implicitly, preserving good stability properties of the numerical scheme \eqref{CNrelax}. Actually, the second equation in \eqref{CNrelax} produces approximations that coincide with the Crank-Nicolson approximations of the linear Schr\"odinger equation
\begin{equation}
\label{linschr}
\p_t u-\ii\alpha\D u=\ii\lambda V(x,t)u\ \text{ in }\Om\times(0,T],\quad V(x,t):=\Phi^{n+\frac12}(x),\ \text{ in } \Om\times I_n,
\end{equation}
 i.e., with potential $V$ a piecewise constant in time function defined through $\Phi^{n+\frac12},\, 0\le n\le N-1.$
 From this point of view, the relaxation scheme \eqref{CNrelax} may also be regarded as a linearised  Crank-Nicolson  method for the NLS problem \eqref{NLS}. Details on stability and convergence results for scheme \eqref{CNrelax} can be found in \cite{Besse}. In particular, method \eqref{CNrelax} is expected to be second order accurate, \cite{Besse}.
 
 Moreover, the relaxation scheme \eqref{CNrelax} satisfies $\|U^n\|=\|U^0\|,\, 0\le n\le N,$ which is the discrete analogue of the mass conservation \eqref{cl1}. This can easily be proven by taking the $L^2-$inner product with $U^{n+\frac12}$ in the second equation of \eqref{CNrelax} and then real parts and by noting that $\Phi^{n+\frac12},\, 0\le n\le N-1,$ is always real. Also, as it has been proven in \cite{Besse}, for $p=1$ and uniform time steps, the numerical scheme \eqref{CNrelax} also satisfies a discrete analogue of the energy conservation \eqref{cl2}. 

Our goal is to derive optimal order a posteriori error bounds for \eqref{NLS} when it is discretized in time by \eqref{CNrelax}.  In order to achieve this, we first define the continuous in time approximation $U(t)$ to $u(t)$, $0\le t\le T$, by linearly interpolating between the nodal values $U^n$ and $U^{n+1}$, i.e., let $U:[0,T]\to H_0^1(\Om)$ be defined as
\begin{equation}
\label{linearint}
U(t):=\ell_0^n(t)U^n+\ell_1^n(t)U^{n+1}=U^{n+\frac12}+(t-t_{n+\frac12})\pa U^n,\quad t\in I_n,
\end{equation}
with $\ell_0^n(t):=\frac{t_{n+1}-t}{k_n}$ and $\ell_1^n(t):=\frac{t-t_n}{k_n}$, $0\le n\le N-1.$

It is clear that  if $u(t_n)-U^n=\mathcal{O}(k^2),\, 0\le n\le N$, then for $t\in[0,T]$, we have  $u(t)-U(t)=\mathcal{O}(k^2),$  as well. However, it is well known, cf., e.g., \cite{AMN1, D, LPP}, that using $U(t)$ in the a posteriori error analysis leads to estimates of first instead of second order of accuracy, which is the expected order of accuracy of method \eqref{CNrelax}. For this reason we shall use a reconstruction $\hat U$ of $U$, \cite{AMN1}. An alternative reconstruction  for $U$ and the scheme \eqref{CNrelax} can be proposed by following similar arguments as in \cite{LPP}.
%
\subsection{Time-reconstruction and its properties}
We introduce now a reconstruction $\hat U$ of $U$. To this end, $\hat U(t)$, $t\in[0,T]$, is defined to be the piecewise quadratic polynomial  in $t$,
\begin{equation}
\label{CNrelaxrec}
\hat U(t):=U^n+\ii\alpha\int_{t_n}^t\D U(s)\,ds+\ii\lambda\int_{t_n}^t\Phi^{n+\frac12}U(s)\,ds,\quad t\in I_n.
\end{equation}
As expected  $\hat U$ coincides with the Crank-Nicolson reconstruction of \cite{AMN1} (see also \cite{Kyza}) for the linear Schr\"odinger equation \eqref{linschr}. It is also worth mentioning that $\hat U(t)\in H_0^1(\Om),\, t\in[0,T]$, provided compatibility conditions on the initial data, cf., \cite{PhDKyza,Kyza}. However, as we will see on Section~\ref{fullydiscr}, where fully discrete schemes are studied, $\hat U(t)\in H_0^1(\Om),\, t\in[0,T],$ automatically, without further assumptions on the initial data. Therefore, from now on we assume that $\hat U$ belongs to $H_0^1(\Om).$
\begin{proposition}[properties of $\hat{U}$]\label{proprec}
The reconstruction $\hat U$ defined in \eqref{CNrelaxrec} is equivalently written as 
\begin{equation}
\label{CNrelaxrec2}
\hat U(t)= U^n+\frac\ii2(t-t_n)(\alpha\D+\lambda\Phi^{n+\frac12})\left(U(t)+U^n\right),\quad t\in I_n.
\end{equation}
Moreover, it satisfies
\begin{equation}
\label{pertPDE}
\p_t\hat U-\ii\alpha\D U=\ii\lambda\Phi^{n+\frac12}U\ \text{ in } I_n
\end{equation}
and $\hat U(t_n^+)=U^n,$ $\hat U(t_{n+1})=U^{n+1},\, 0\le n\le N-1$; in particular, $\hat U$ is a time-continuous function.
\end{proposition}
\begin{proof}
First we note that $\int_{t_n}^tU(s)\,ds=\frac12(t-t_n)\left(U(t)+U^n\right)$, and therefore, \eqref{CNrelaxrec2} is obtained from \eqref{CNrelaxrec}, whilst \eqref{pertPDE} is directly obtained by differentiation in time of \eqref{CNrelaxrec}. Finally, $U(t_n^+)=U^n$ is obvious once more from \eqref{CNrelaxrec}, while \eqref{CNrelaxrec2} gives, for $t=t_{n+1},$
$$\hat{U}(t_{n+1})=U^n+\ii\alpha k_n\D U^{n+\frac12}+\ii\lambda k_n\Phi^{n+\frac12}U^{n+\frac12},$$
whence $\hat U(t_{n+1})=U^{n+1}$ follows from \eqref{CNrelax}.
\end{proof}

The a posteriori quantity $\hat r$, by which $\hat U$ misses satisfying the NLS equation in \eqref{NLS}, is the residual and it is defined as 
\begin{equation}
\label{residual1}
\hat r:=\p_t \hat U-\ii\alpha\D\hat U-\ii\lambda f(\hat U)\ \text{ in } I_n.
\end{equation}
By virtue of \eqref{pertPDE}, we have that  $\hat r $ is written as 
\begin{equation}
\label{residual2}
\hat r=-\ii\alpha\D(\hat U-U)-\ii\lambda\big(f(\hat U)-f(U)\big)-\ii\lambda(|U|^{2p}-\Phi^{n+\frac12})U\ \text{ in } I_n.
\end{equation}
We would like $\hat r$ to  decrease as fast as the order of the method, i.e., we would like $\hat r$ to be of second order of accuracy, provided that $u(t_n)-U^n=\mathcal{O}(k^2),\, 0\le n\le N.$ Towards that direction, in the next proposition we compute the difference $\hat U-U$.
\begin{proposition}[the difference $\hat U-U$]\label{diffprop}
The difference $\hat U-U$ can be expressed as 
\begin{equation}
\label{difference1}
(\hat U-U)(t)=-\frac{\ii}2(t-t_n)(t_{n+1}-t)(\alpha\D+\lambda\Phi^{n+\frac12})\pa U^n,\quad t\in I_n.
\end{equation}
\end{proposition}
\begin{proof}
We compute $\hat U-U$ by subtracting \eqref{linearint} from \eqref{CNrelaxrec2} and we use \eqref{CNrelax}.
The expression \eqref{difference1} follows then by basic algebraic manipulations. 
\end{proof}

Given that the relaxation scheme \eqref{CNrelax} produces second order approximations to the exact solution at the nodes $t_n,\, 0\le n\le N$, we conclude that $\hat U-U$ is expected to be of second order of accuracy. Also, from the first equation of \eqref{CNrelax}, \eqref{approx1} and \eqref{mildstepcons}, we expect that $\Phi^{n+\frac12}=|u(\tha)|^{2p}+\mathcal{O}(k^2),\, 0\le n\le N-1,$ as $|z|^{2p}$ is a locally Lipschitz continuous function for $p\ge\frac12.$ Using all these and the fact that $f(z)=|z|^{2p}z$ is also a locally Lipschitz function for $p\ge\frac12$, and resorting back to \eqref{residual2}, we deduce that $\hat r$ is expected to be second order accurate.
%
\subsection{Variational formulation \& Error equation}
Problem \eqref{NLS} is equivalently written,  in variational form, for $t\in[0,T]$, as
\begin{equation}
\label{NLSvar} \left \{
\begin{aligned}
&\langle \p_tu(t),v\rangle+\mathrm{i}\alpha\langle\nabla u(t),\nabla v\rangle=\ii\la \langle f\left(u(t)\right),v\rangle,    &&\quad\forall v\in H_0^1(\Om), & \\
&u(\cdot,0)=u_0 &&\quad\mbox{in ${\varOmega}$}.&
\end{aligned}
\right.
\end{equation}
Similarly, the reconstruction $\hat U$ defined in \eqref{CNrelaxrec} satisfies, for $t\in I_n$, the problem
\begin{equation}
\label{recvar} \left \{
\begin{aligned}
&\langle \p_t\hat U(t),v\rangle+\mathrm{i}\alpha\langle\nabla \hat U(t),\nabla v\rangle=\ii\la \langle f\left(\hat U(t)\right),v\rangle+\langle\hat r(t),v\rangle,    &&\quad\forall v\in H_0^1(\Om), & \\
&u(\cdot,0)=u_0 &&\quad\mbox{in ${\varOmega}$},&
\end{aligned}
\right.
\end{equation}
cf., \eqref{residual1}.

We denote by $\hat e:=u-\hat U$ the error between the exact solution $u$ and the reconstruction $\hat U$. Then subtracting the first equation of \eqref{recvar} from the first equation of \eqref{NLSvar} we derive the basic error equation
\begin{equation}
\label{errorbasic}
\langle \hat e_t(t),v\rangle+\mathrm{i}\alpha\langle\nabla \hat e(t),\nabla v\rangle=\ii\la \langle f\left(u(t)\right)-f\left(\hat U(t)\right),v\rangle-\langle\hat r(t),v\rangle,   \ \forall v\in H_0^1(\Om),\, t\in I_n,
\end{equation} 
with  $\hat e(0)=0.$
Taking $v=\hat e$ in the error equation \eqref{errorbasic} and then real parts reveals
\begin{equation}
\label{error2}
\frac12\frac d{dt}\|\hat e(t)\|^2=\la\Rea\ii\langle f(u)-f(\hat U),\hat e\rangle(t)-\Rea\langle\hat r,\hat e\rangle(t),\quad t\in I_n.
\end{equation}

In Section~\ref{apostimdis} we analyse how we treat the term $\Rea\ii\langle f(u)-f(\hat U),\hat e\rangle$ on the right-hand side of \eqref{error2} to derive an a posteriori error bound for \eqref{NLS}.
%
\section{A posteriori error control for time discrete schemes}\label{apostimdis}
In view of \eqref{error2} it is evident that we have to handle the term $\Rea\ii\langle f(u)-f(\hat U),\hat e\rangle,$ which arises due to the nonlinear nature of the problem. To this direction we will use the Gagliardo-Nirenberg inequality \eqref{GN}, the conservation laws \eqref{cl1} and \eqref{cl2} and some other ingredients  presented in the next subsection.
\subsection{Main Ingredients}
Firstly, we compute the Gateaux derivative of $f(z)=|z|^{2p}z$ in the direction of $v$:
\begin{equation}
\label{Gd}
\begin{aligned}
D_vf(z):&=\lim_{\varepsilon\to 0}\frac1\varepsilon\left(f(z+\varepsilon v)-f(z)\right)
=2p|z|^{2(p-1)}z\Rea(\bar z v)+|z|^{2p}v.
\end{aligned}
\end{equation}
Thus using the Mean Value Theorem, we can write, for $z_1, z_2\in\mathbb{C}$, that
\begin{equation}
\label{MTV}
f(z_1)-f(z_2)=D_vf(z),
\end{equation}
where $z$ is a convex combination of $z_1$ and $z_2$. In other words, there exists $s\in[0,1]$ such that
\begin{equation}
\label{convex}
z=sz_1+(1-s)z_2.
\end{equation}

Hence, plugging \eqref{Gd} into \eqref{MTV} with $z_1=z_1(t)$, $z_2=z_2(t)$ and $v=v(t)=z_1(t)-z_2(t)$, we obtain, for $t\in[0,T],$
\begin{equation}
\label{NLdiff}
\begin{aligned}
\big(f(z_1)-f(z_2)\big)(t)=\Big[&2p|sz_1+(1-s)z_2|^{2(p-1)}(sz_1+(1-s)z_2)\Rea\left((s\bar z_1+(1-s)\bar{z}_2)v\right)\\&
+|sz_1+(1-s)z_2|^{2p}v\Big] (t),
\end{aligned}
\end{equation}
with $s(t)\in[0,1]$ for $t\in[0,T].$ In \eqref{NLdiff} we also used \eqref{convex}. Therefore, in the case of $z_1=u$, $z_2=\hat U$ and $v=\hat e$, and by noting that $su+(1-s)\hat U=s\hat e+\hat U$ we derive
\begin{equation}
\label{nlp1}
\begin{aligned}
&|\Rea\ii\langle f(u)-f(\hat U),\hat e\rangle(t)|\\&
=2p|\Rea\ii\langle|su+(1-s)\hat U|^{2(p-1)}(su+(1-s)\hat U)\Rea\left((s\bar u+(1-s)\bar{\hat U})\hat e\right),\hat e\rangle(t)|\\&\le2p\langle|s\hat e+\hat U|^{2p}\hat e,\hat e\rangle(t).
\end{aligned}
\end{equation}

Next to handle the term $|s\hat e+\hat U|^{2p}$ in \eqref{nlp1} we use the standard inequality
\begin{equation}
\label{SI}
(a+b)^q\le \gamma(q)(a^q+b^q)\quad \text{with}\quad \gamma(q):=\left\{\begin{aligned}&2^{q-1},\ \hspace{0.15cm}q\ge 1 \\&1,\ 0\le q\le1,\end{aligned}\right.
\end{equation}
valid for any $a,b\ge 0$ and $q\ge 0$. Inequality \eqref{SI} can easily be proven for $q\ge 1,$ by noticing that function $g(x)=x^q$ is convex for $x\ge 0$ and $q\ge 1$ (see also \cite{Rudin}). For $0\le q\le 1,$ inequality \eqref{SI} holds because in this case $g(x)=x^q$ is sublinear.

Applying \eqref{SI} to $|s\hat e+\hat U|^{2p}$ we get $|s\hat e+\hat U|^{2p}\le 2^{2p-1}(|\hat e|^{2p}+|\hat U|^{2p})$; recall that $0\le s\le 1$ and $p\ge\frac12.$ Hence \eqref{nlp1} becomes
\begin{equation}
\label{nlp2}
|\Rea\ii\langle f(u)-f(\hat U),\hat e\rangle(t)|\le 2^{2p}p\left(\|\hat e(t)\|^{2p+2}_{L^{2p+2}}+\|\hat U(t)\|^{2p}_{L^\infty}\|\hat e(t)\|^2\right).
\end{equation}
To estimate $\|\hat e(t)\|^{2p+2}_{L^{2p+2}}$ we refer to the Gagliardo-Nirenberg inequality \eqref{GN}:
\begin{equation}
\label{errorest1}
\begin{aligned}
\|\hat e(t)\|^{2p+2}_{L^{2p+2}}&\le B\|\nabla\hat e(t)\|^{pd}\|\hat e(t)\|^{p(2-d)+2}
\\&
=B\|\nabla\hat e(t)\|^{pd}\|\hat e(t)\|^{p(2-d)}\|\hat e(t)\|^2,
\end{aligned}
\end{equation}
with $B:=\beta^{2p+2}.$

Since $d\le 2$ and $\frac 12\le p\le\frac 2d,$ we have that $p(2-d)\ge 0$ and thus by \eqref{SI} we estimate
\begin{equation}
\label{errorest2}
\begin{aligned}
\|\hat e(t)\|^{p(2-d)}&\le \gamma\left(p(2-d)\right)\left(\|u(t)\|^{p(2-d)}+\|\hat U(t)\|^{p(2-d)}\right)
\\&
= \gamma\left(p(2-d)\right)\left(\|u_0\|^{p(2-d)}+\|\hat U(t)\|^{p(2-d)}\right),
\end{aligned}
\end{equation}
where the last equality holds because of \eqref{cl1}. Set $A:=|\la| 2^{2p}p\max\{1,  \gamma\left(p(2-d)\right)B\}.$ Then using \eqref{errorest1} and \eqref{errorest2} in \eqref{nlp2} we deduce, for $\la\ne 0$ \footnote{if $\la=0$ then A:=0, i.e., the problem is reduced to the linear one.} and $t\in[0,T],$ 
\begin{equation}
\label{nlp3}
|\Rea\ii\langle f(u)-f(\hat U),\hat e\rangle(t)|\le\frac A{|\la|}\left(\left(\|u_0\|^{p(2-d)}+\|\hat{U}(t)\|^{p(2-d)}\right)\|\nabla\hat e(t)\|^{pd}+\|\hat U(t)\|^{2p}_{L^\infty}\right)\|\hat e(t)\|^2,
\end{equation}
whereas returning back to the error equation \eqref{error2} we conclude that it can be written, for $t\in I_n$, as 
\begin{equation}
\label{error3}
\frac d{dt}\|\hat e(t)\|\le A \left(\left(\|u_0\|^{p(2-d)}+\|\hat{U}(t)\|^{p(2-d)}\right)\|\nabla\hat e(t)\|^{pd}+\|\hat U(t)\|^{2p}_{L^\infty}\right)\|\hat e(t)\|+\|\hat r (t)\|.
\end{equation}
From \eqref{error3} it is clear that we can bound the $L^\infty(L^2)-$norm of $\hat e$, using Gronwall's inequality, as long as we have an estimation for $\|\nabla\hat e(t)\|.$ We do this in the next subsection.

\begin{remark}[the $3d$ case]\upshape
If $d=3$, we have that $p(2-d)<0$ and \eqref{errorest2} fails. This is the reason why the three dimensional spatial case cannot be included in the analysis of this paper. 
\end{remark}
%
\subsection{A posteriori error bound in the $L^\infty(L^2)-$norm}
In order to state the main theorem of this section, we first estimate a posteriori  the gradient term in \eqref{error3}. Since $pd\ge \frac 12>0,$ \eqref{SI} implies
\begin{equation}
\label{graderror4}
\|\nabla\hat e(t)\|^{pd}\le \gamma(pd)\left(\|\nabla\hat U(t)\|^{pd}+\|\nabla u(t)\|^{pd}\right).
\end{equation}
In the next lemma we bound a posteriori the term $\|\nabla u(t)\|.$
\begin{lemma}[a posteriori bound for $\|\nabla u(t)\|$]
For $t\in[0,T],$ the following estimate holds true
\begin{equation}
\label{estgrad}
\|\nabla u(t)\|\le G(u_0)
\end{equation}
with
\begin{equation}
\label{boundgrad}
G(u_0):=\left\{\begin{aligned}&\left(\|\nabla u_0\|^2-\frac\la{\alpha(p+1)}\|u_0\|^{2p+2}_{L^{2p+2}}\right)^{1/2},\quad \la\le 0\ \\&\left(\|\nabla u_0\|^{2-pd}+\frac\la{\alpha(p+1)}\left(\beta^{2p+2}\|u_0\|^{p(2-d)+2}-\frac{\|u_0\|^{2p+2}_{L^{2p+2}}}{\|\nabla u_0\|^{pd}}\right)
\right)^{1/(2-pd)},\\&\hspace{10.0cm} \la>0,\, \frac12\le p<\frac2d \\&\frac1{\left(1-\Gamma(u_0)\right)^{1/2}}\left(\|\nabla u_0\|^2-\frac\la{\alpha(p+1)}\|u_0\|^{2p+2}_{L^{2p+2}}\right)^{1/2},\quad pd=2,\, \la>0,\,\Gamma(u_0)<1,\end{aligned}\right.
\end{equation}
where $\beta$ denotes the constant in the Gagliardo-Nirenberg inequality \eqref{GN} and $\Gamma(u_0)$ is given in \eqref{initialgrowth}.
\end{lemma}
\begin{proof}
We divide the proof into three parts.
First we assume that $\la\le 0$. Then from the energy conservation \eqref{cl2} we readily obtain
$$\|\nabla u(t)\|^2\le\|\nabla u_0\|^2-\frac\la{\alpha(p+1)}\|u_0\|^{2p+2}_{L^{2p+2}}.$$
Next we assume that $\la>0$ and $\frac12\le p\le\frac2d.$ Using the conservation laws \eqref{cl1}-\eqref{cl2} and invoking the Gagliardo-Nirenberg inequality \eqref{GN} to the term $\|u(t)\|^{2p+2}_{L^{2p+2}}$ leads to
\begin{equation}
\label{estgrad1}
\begin{aligned}
\|\nabla u(t)\|^{pd}&\left(\|\nabla u(t)\|^{2-pd}-\frac{B\la}{\alpha(p+1)}\|u_0\|^{p(2-d)+2}\right)\\&
\le\|\nabla u_0\|^{pd}\left(\|\nabla u_0\|^{2-pd}-\frac{\la}{\alpha(p+1)}\frac{\|u_0\|^{2p+2}}{\|\nabla u_0\|^{pd}}
\right).
\end{aligned}
\end{equation}
Hence if $\|\nabla u(t)\|>\|\nabla u_0\|$ we must have, from \eqref{estgrad1}, that
\begin{equation}
\label{estgrad2}
\|\nabla u(t)\|\le\left(\|\nabla u_0\|^{2-pd}-\frac\la{\alpha(p+1)}\frac{\|u_0\|^{2p+2}_{L^{2p+2}}}{\|\nabla u_0\|^{pd}}+\frac{B\la}{\alpha(p+1)}\|u_0\|^{p(2-d)+2}\right)^{1/(2-pd)},
\end{equation}
because $2-pd>0$. Note now that \eqref{GN} implies $\displaystyle\frac{\|u_0\|^{2p+2}_{L^{2p+2}}}{\|\nabla u_0\|^{pd}}\le B\|u_0\|^{p(2-d)+2}.$ This observation in combination with \eqref{estgrad2} give \eqref{estgrad} for $\la>0$ and $\frac12\le p<\frac 2d.$

Finally we consider the case $\la>0$, $\Gamma(u_0)<1$ and $pd=2$, i.e., the critical exponent. Using the same argumentation as for the derivation of \eqref{estgrad1} we obtain
$$\left(1-\Gamma(u_0)\right)\|\nabla u(t)\|^2\le\|\nabla u_0\|^2-\frac\la{\alpha(p+1)}\|u_0\|^{2p+2}_{L^{2p+2}},$$
which implies \eqref{estgrad} for this case as well, and the proof is complete.
\end{proof}
\begin{remark}[the critical exponent]\upshape
In view of \eqref{boundgrad} it is obvious that $G(u_0)$ blows up when $\Gamma(u_0)\to 1$. This is reasonable as estimate \eqref{estgrad} is proven uniformly for all $t,$ and for $\Gamma(u_0)\ge 1$ the $H^1-$norm of $u$ may blow up in finite time. 
\end{remark}

\begin{remark}[computational cost of $G(u_0)$]\upshape
Although $G(u_0)$ in \eqref{boundgrad} seems complicated, it is a global quantity and it is computed only once numerically.
\end{remark}

If we define
\begin{equation}
\label{recfunctional}
\Hf(\hat{U},u_0;t):=A\left(\gamma(pd)\left(\|u_0\|^{p(2-d)}+\|\hat U(t)\|^{p(2-d)}\right)\left(G(u_0)^{pd}+\|\nabla\hat U(t)\|^{pd}\right)+\|\hat U(t)\|^{2p}_{L^\infty}\right),
\end{equation}
we derive in light of \eqref{error3}, \eqref{graderror4} and \eqref{estgrad}
\begin{equation}
\label{error4}
\frac d{dt}\|\hat e(t)\|\le\Hf(\hat U, u_0;t)\|\hat e(t)\|+\|\hat r(t)\|,\quad t\in I_n.
\end{equation}
Thus, using Gronwall's inequality and keeping in mind that $\hat e=u-\hat U$ is a time-continuous function with $\hat e(0)=0$, we arrive at the main theorem of the section.
\begin{theorem}[a posteriori error estimate in $L^\infty(L^2)$]\label{apostdiscth}
The following local and global a posteriori error estimates in the $L^\infty(L^2)-$norm are valid:
\begin{equation}
\label{localest1}
\sup_{t\in I_n}\|(u-\hat U)(t)\|\le\exp\left(\int_{I_n}\Hf(\hat U, u_0;t)\,dt\right)\left(\|(u-\hat U)(t_n)\|+\int_{I_n}\|\hat r(t)\|\,dt\right),\ 0\le n\le N-1,
\end{equation}
and
\begin{equation}
\label{globalest1}
\|(u-\hat U)(t)\|\le\exp\left(\int_0^t\Hf(\hat U,u_0;s)\,ds\right)\int_0^t\|\hat r(s)\|\,ds, \quad t\in [0,T],
\end{equation}
where $\Hf(\hat U, u_0;t)$ is given in \eqref{recfunctional}, $u$ is the solution of \eqref{NLS}, $\hat U$ is the reconstruction \eqref{CNrelaxrec2} for the relaxation Crank-Nicolson-type scheme \eqref{CNrelax} and $\hat r$ is the residual given in \eqref{residual2}.\qed
\end{theorem}

\begin{remark}[focusing NLS]\upshape
Estimates \eqref{localest1} and \eqref{globalest1} blow up exponentially when $\Gamma(u_0)\to 1$, because $G(u_0)$ (and thus $\Hf(\hat U, u_0;t)$) blows up in this case (cf., \eqref{boundgrad}). Thus the methodology of this section cannot be generalised to the focusing NLS equation with the aim to control the error close to the blowup time. In contrast to the corresponding parabolic equation with blowup, cf., \cite{CGKM}, energy methods are not appropriate for the focusing NLS with blowup and other techniques should be applied for the a posteriori error analysis of these equations.

\end{remark}

\begin{remark}[local estimate \eqref{localest1}]\upshape
Estimate \eqref{localest1} may be used for the proposition of an efficient adaptive algorithm as in \cite{CGKM}. This estimate is more appropriate for adaptivity due its local nature. Note that $\|(u-\hat U)(t_{n})\|$ comes from the previous time step, so if we have $\displaystyle\max_{t\in I_{n-1}}\|(u-\hat U)(t)\|\le \text{TOL}$, where $\text{TOL}$ is a given tolerance, then $\text{TOL}$ can replace $\|(u-\hat U)(t_{n})\|$ in \eqref{localest1}.
\end{remark}
%
\subsection{An improved estimate for the one-dimensional case}
In this subsection we consider the one spatial dimension, $d=1$. Let $\eta(t)$ denote the right-hand side in \eqref{globalest1} to the power $p$. Then in view of \eqref{globalest1}, \eqref{errorest1} takes the form
$$\|\hat e(t)\|^{2p+2}_{L^{2p+2}}\le B\|\nabla\hat e(t)\|^p\eta(t)\|\hat e(t)\|^2.$$
Plugging the above estimate in \eqref{nlp2} and using \eqref{graderror4} and \eqref{estgrad} we deduce
\begin{equation*}
\begin{aligned}
|\Rea\ii\langle f(u)-f(\hat U),\hat e\rangle(t)|\le&\frac A{|\la|}\left(\eta(t)\left(G(u_0)^p+\|\nabla\hat U(t)\|^p\right)+\|\hat U(t)\|^{2p}_{L^\infty}\right)\|\hat e(t)\|
+\|\hat r(t)\|, \quad t\in I_n.
\end{aligned}
\end{equation*}
Hence the error equation \eqref{error2} takes now the form
$$\frac d{dt}\|\hat e(t)\|\le \mathcal{K}(\hat U,u_0;t)\|\hat e(t)\|+\|\hat r(t)\|, \quad t\in I_n,$$
with
\begin{equation}
\label{recfunctional2}
\mathcal{K}(\hat U, u_0; t):= A\left(\eta(t)\left(G(u_0)^p+\|\nabla\hat U(t)\|^p\right)+\|\hat U(t)\|^{2p}_{L^\infty}\right).
\end{equation}
Thus Gronwall's inequality leads to:
\begin{theorem}[improved a posteriori error estimate in $L^\infty(L^2)$ for $d=1$]\label{impth}
In the case $d=1$, the following estimates hold:
\begin{equation*}
\sup_{t\in I_n}\|(u-\hat U)(t)\|\le\exp\left(\int_{I_n}\mathcal{K}(\hat U, u_0;t)\,dt\right)\left(\|(u-\hat U)(t_n)\|+\int_{I_n}\|\hat r(t)\|\,dt\right),\ 0\le n\le N-1,
\end{equation*}
and
\begin{equation}
\label{imprestgl}
\|(u-\hat U)(t)\|\le\exp\left(\int_0^t\mathcal{K}(\hat U, u_0;s)\,ds\right)\int_0^t\|\hat r(s)\|\,ds,
\end{equation}
where $u$ is the solution of \eqref{NLS}, $\hat U$ is the reconstruction \eqref{CNrelaxrec2} for the relaxation Crank-Nicolson-type scheme \eqref{CNrelax}, $\hat r$ denotes the residual \eqref{residual2} and $\mathcal{K}(\hat U, u_0;s)$ is given by \eqref{recfunctional2}.\qed
\end{theorem}

\begin{remark}[why is \eqref{imprestgl} an  improved estimate?]\upshape\label{impest}
Note that since $\eta(t)=\mathcal{O}(k^{2p}),$ $t\in[0,T],$ we have that 
$$\int_0^t \mathcal{K}(\hat U, u_0;s)\,ds=\mathcal{O}(k^{2p})+A\int_0^t\|\hat U(s)\|^{2p}_{L^\infty}\,ds.$$
From this point of view, \eqref{imprestgl} is an improved estimate compared to \eqref{globalest1}.
\end{remark}
%
%
\section{Fully discrete schemes}\label{fullydiscr}
In this section we study fully discrete schemes.  A Galerkin-type finite element method is used for the spatial discretization 
while for the discretization in time we use the relaxation Crank-Nicolson-type method \eqref{CNrelax}. We begin the section with some notation and the introduction of the fully discrete scheme.
%
\subsection{Notation \& The method}
We keep the same notation as in time discretization section for the partition of $[0,T]$, i.e., we denote by $k_n$ and $I_n$, $0\le n\le N-1$, the variable time steps and subintervals, respectively. For the spatial discretization, we follow the same notation as in \cite{KK}. More precisely, we consider a family of conforming, shape regular triangulations $\{\T_n\}_{n=0}^N$ of $\Om$ (for $d=1$, the elements of $\T_n$ are just finite intervals). We additionally assume that each triangulation $\T_n,\, 1\le n\le N,$ is a refinement of a macro-triangulation of $\Om$ and that every two consecutive triangulations $\T_n$ and $\T_{n+1}$, $0\le n\le N-1$, are compatible. We refer to \cite{DLM,LM} for precise definitions on these properties. 

For an element $K\in\T_n$, we denote its boundary by $\p K$ and by $h_K$ its diameter. Let also $\underline{h}_n:=\min_{K\in\T_n}h_K.$ By $h_n$ we denote the local mesh-size function on $\T_n$, defined as 
$$h_n(x):=h_{K}\ \text{ for } K\in\T_n\text{ and } x\in K.$$
 
 Let  $\varSigma_n(K)$ be the set of internal sides of $K\in\T_n$ (points for $d=1$, edges for $d=2$) and define $\varSigma_n:=\bigcup_{K\in\T_n}\varSigma_n(K).$ To any side $e\in\varSigma_n$, we associate a unit vector $\bm{n}_e$ on $e$ and for $x\in e$ and a function $v$, we define 
$$J[\nabla v](x):=\lim_{\delta\to 0}\Big[\nabla v(x+\delta\bm{n}_e)-\nabla v(x-\delta\bm{n}_e)\Big]\cdot\bm{n}_e.$$

To each triangulation $\T_n,$ we associate the finite element space $\V^n$,
$$\V^n:=\{V_n\in H_0^1(\Om):\forall K\in\T_n,\, V_n|_K\in\mathbb{P}^r\},$$
where $\mathbb{P}^r$ denotes the space of polynomials in $d$ variables of degree at most $r$.

With $\widehat{\T}_{n+1}:=\T_{n+1}\wedge\T_n$ we denote the finest common coarsening triangulation of $\T_{n+1}$ and $\T_n$, by $\widehat{h}_{n+1}$ its local mesh-size function and by $\widehat{\V}^{n+1}:=\V^{n+1}\bigcap\V^n$ its corresponding finite element space. Finally, let  $\check{\varSigma}_{n+1}:=\varSigma_{n+1}\bigcup\varSigma_{n},$ and for
$K\in\widehat{\T}_{n+1},$ let $\check{\varSigma}_K^{n+1}:=\check{\varSigma}_{n+1}\bigcap
K,$ where the element $K\in\widehat{\T}_{n+1}$ is taken to be closed.

To introduce a fully discrete method, we will also need the definitions of the $L^2-$projection and of the discrete laplacian onto $\V^n$.
To this end, the $L^2-$projection $\Pro^n:L^2\to\V^n$ is defined as 
$$\langle\Pro^n v, V_n\rangle=\langle v,V_n\rangle,\quad\forall V_n\in\V^n,$$
and every $v\in L^2(\Om)$. Moreover, the discrete laplacian $-\D^n: H_0^1(\Om)\to\V^n$ is defined as
\begin{equation}
\label{discrlapl}
\langle-\D^n v, V_n\rangle=\langle\nabla v,\nabla V_n\rangle,\quad\forall V_n\in\V^n,
\end{equation}
and every $v\in H_0^1(\Om).$

With the notation and definitions so far, we can now define the modified relaxation Crank-Nicolson-Galerkin-type fully discrete scheme. For $0\le n\le N,$ we seek approximations $U^n\in\V^n$ to $u(t_n)$ such that, for $0\le n\le N-1$,
\begin{equation}
\label{CNrelaxfull} \left \{
\begin{aligned}
& \frac{k_{n-1}}{k_n+k_{n-1}}\Phi^{n+\frac12}+\frac{k_n}{k_n+k_{n-1}}\Pro^{n+1}\Phi^{n-\frac12}=\Pro^{n+1}\left(|U^n|^{2p}\right),    \\
&\frac{U^{n+1}-\Pro^{n+1}U^n}{k_n}-\ii\alpha\frac{\D^{n+1}U^{n+1}+\Pro^{n+1}\D^nU^n}2=\ii\lambda\Pro^{n+1}\left(\Phi^{n+\frac12}U^{n+\frac12}\right),
\end{aligned}
\right.
\end{equation}
with $\Phi^{-\frac12}=\Pro^0\left(|u_0|^{2p}\right)$, $U^0=\Pro^0 u_0.$ Scheme \eqref{CNrelaxfull} is not the standard finite element scheme for \eqref{CNrelax}. The above  modified scheme was introduced by B\"{a}nsch,  Karakatsani \& Makridakis in \cite{BKM} for the a posteriori analysis of the heat equation with mesh change. Its main advantage is that it avoids the existence of the term $\|(\D^{n+1}-\D^n)U^n\|$ in the a posteriori error analysis, which oscillates when there is mesh change. This schemes was also used recently in \cite{KK} for linear Schr\"odinger equations. 
%

\subsection{Space reconstruction}
The main tool leading to a posteriori error estimates of optimal order in space in the $L^\infty(L^2)-$norm via energy techniques is the elliptic reconstruction. It was originally introduced by Makridakis \& Nochetto in \cite{MN1} for  finite element semidiscrete schemes for the heat equation.

For $V_n\in\V^n$, the \emph{elliptic reconstruction} $\R^n V_n\in H_0^1(\Om)$ of $V_n$ is defined to be the unique solution of the elliptic equation
\begin{equation}
\label{ellrec}
\langle\nabla\R^n V_n,\nabla v\rangle=\langle-\D^nV_n,v\rangle,\quad \forall v\in H_0^1(\Om).
\end{equation}

Using the elliptic reconstruction, we define the \emph{space reconstruction} $\omega: [0,T]\in H_0^1(\Om)$ of the piecewise linear interpolated $U$ (see \eqref{linearint}) as 
\begin{equation}
\label{spacerec}
\omega(t):=\ell_0^n(t)\R^nU^n+\ell_1^n(t)\R^{n+1}U^{n+1},\quad t\in I_n.
\end{equation}
The space reconstruction \eqref{spacerec} will allow us to handle efficiently a posteriori the spatial error using the elliptic theory. More precisely, using $\omega$ in the analysis below, terms of the form $\|(\R^n-\II)V_n\|,$ $\|(\R^n-\II)V_n\|_{L^\infty}$ and $\|(\R^{n+1}-\II)V_{n+1}-(\R^n-\II)V_n\|$ will appear. To estimate these terms, we will use residual-type elliptic error estimators.

To this end, for a given $V_n\in\V^n$, $0\le n\le N$, we define the following $L^2-$ and $L^\infty-$residual-type estimators:
\begin{equation}
\label{l2ell1}
\begin{aligned}
\h(V_n):=\bigg\{\sum_{K\in\T_n}\Big(\|h_K^2(\D-\D^n)V_n\|^2_{L^2(K)}+\|h_K^{\frac  32}J[\nabla V_n]\|^2_{L^2(\p K)}\Big)\bigg\}^{\frac 12},
\end{aligned}
\end{equation}
and 
\begin{equation}
\label{linfell}
\begin{aligned}
\hin(V_n):=\max_{K\in\T_n}\bigg\{\|h_K^2(\D-\D^n)V_n\|_{L^\infty(K)}+\|h_KJ[\nabla V_n]\|_{L^\infty(\p K)}\bigg\},
\end{aligned}
\end{equation}
where, for $p=1,\infty$, $\|\cdot\|_{L^p(K)}$ and $\|\cdot\|_{L^p(\p K)}$ denote the $L^p-$norm in $K$ and on $\p K$, respectively. In the one-dimensional case, $d=1,$ the term with the discontinuities in \eqref{l2ell1} and \eqref{linfell} vanishes. For $V^{n+1}\in\V_{n+1}$ and $V_n\in\V^n$, $0\le n\le N$, we also define
\begin{equation}
\label{l2ell2}
\begin{aligned}
\het(V_{n+1},V_n):=\bigg\{\sum_{K\in\widehat\T_n}\Big(\|h_K^2\big[(\D-\D^{n+1})V_{n+1} & -(\D-\D^n)V_n\big]\|^2_{L^2(K)} 
\\*[-15pt] &+\| h_{K}^{\frac 32}J[\nabla V_{n+1}-\nabla V_n]\|^2_{L^2(\check{\varSigma}_K^{n+1})}\Big)\bigg\}^{\frac 12}.
\end{aligned}
\end{equation}

In view of the definitions \eqref{l2ell1} and \eqref{l2ell2}, the next lemma is standard. For its proof we refer, for example, in \cite{MN1,LM} (for the case $d=1$, see \cite{PhDKyza}).
\begin{lemma}
For all $V_n\in\V^n,\, 0\le n\le N,$ we have
\begin{equation}
\label{resest1}
\|(\R^n-\II)V_n\|\le C_2\h(V_n),
\end{equation}
where $C_2$ depends only on $\Om$ and the shape regularity of the family of triangulations $\{\T_n\}_{n=0}^N.$ Furthermore, for all $V_{n+1}\in\V^{n+1}$ and $V_n\in\V^n$, $0\le n\le N-1$, it holds
\begin{equation}
\label{resest2}
\|(\R^{n+1}-\II)V_{n+1}-(\R^n-\II)V_n\|\le\widehat{C}_2\het(V_{n+1},V_n),
\end{equation}
where $\widehat{C}_2$ depends only on $\Om$, the shape regularity of the triangulations, and the number of refinement steps necessary to pass from $\T_n$ to $\T_{n+1}.$ \qed 
\end{lemma}

A similar estimate to \eqref{resest1} holds for $\|(\R^n-\II)V_n\|_{L^\infty}$; for its proof we refer to \cite{NSSV} (for $d=1,$ see \cite{PhDKyza}).
\begin{lemma}
For every $V_n\in\V^n$, $0\le n\le N,$ the following estimate is valid:
\begin{equation}
\label{resest3}
\|(\R^n-\II)V_n\|_{L^\infty}\le C_\infty\left(\ln\underline h_n\right)^2\hin(V_n),
\end{equation}
where $C_\infty$ depends only on $\Om$ and the shape regularity of the family of triangulations. \qed
\end{lemma}
%

\subsection{Time-space reconstruction and its properties}
%
As already discussed in Section~\ref{timediscr} to handle efficiently the time-error due to the discretization with the relaxation Crank-Nicolson scheme, we need  a reconstruction in time. Similarly, for the fully discrete scheme \eqref{CNrelaxfull} we will use a time reconstruction of $\omega$. Thus, we end up with a time-space reconstruction $\hat U:[0,T]\to H_0^1(\Om)$ which is defined as follows:
\begin{equation}
\label{fullyrec1}
\begin{aligned}
\hat U(t):=\R^n U^n&+\frac{t-t_n}{k_n}(\R^{n+1}\Pro^{n+1}U^n-\R^nU^n)+\ii\alpha\int_{t_n}^t\R^{n+1}\Theta(s)\,ds\\&+\ii\lambda\int_{t_n}^t\R^{n+1}\Pro^{n+1}\left(\Phi^{n+\frac12}U(s)\right)\,ds,\quad t\in I_n,
\end{aligned}
\end{equation}
with 
\begin{equation}
\label{thetafunc}
\Theta(t):=\ell_0^n(t)\Pro^{n+1}\D^n U^n+\ell_1^n(t)\D^{n+1}U^{n+1},\quad t\in I_n;
\end{equation}
compare with \eqref{CNrelaxrec}. We choose to use the same notation $\hat U$ for the time-space reconstruction as for the time-reconstruction in time discrete schemes (cf., Sections~\ref{timediscr},~\ref{apostimdis}). This is done in an attempt to simplify the notation as well as for a direct comparison with the time discrete schemes. Note that $\hat U$ in \eqref{fullyrec1} coincides with the time-space reconstruction of \cite{KK} for the linear Schr\"odinger equation \eqref{linschr}.

Denoting by
\begin{equation}
\label{notation1}
W(t):=\ii\alpha\Theta(t)+\ii\lambda\Pro^{n+1}\left(\Phi^{n+\frac12}U(t)\right),\quad t\in I_n,
\end{equation}  
we have that $\hat U$ is equivalently written as
\begin{equation}
\label{fullyrec2}
\hat U(t)=\R^n U^n+\frac{t-t_n}{k_n}(\R^{n+1}\Pro^{n+1} U^n-\R^n U^n)+\int_{t_n}^t\R^{n+1}W(s)\,ds, \quad t\in I_n,
\end{equation}
whilst the second equation in method \eqref{CNrelaxfull} is written as
\begin{equation}
\label{CNrelaxfull2}
\frac{U^{n+1}-\Pro^{n+1}U^n}{k_n}-W(t_{n+\frac12})=0,\quad 0\le n\le N-1.
\end{equation}
For $0\le n\le N-1$, $t\in I_n$, we further define
\begin{equation}
\label{thetader}
\pt\Theta^{n+1}:=\p_t\Theta(t)=\frac{\D^{n+1}U^{n+1}-\Pro^{n+1}\D^nU^n}{k_n}
\end{equation}
and similarly,
\begin{equation}
\label{wder}
\pt W^{n+1}:=\p_t W(t)=\ii\alpha\pt\Theta^n+\ii\lambda\Pro^{n+1}\left(\Phi^{n+\frac12}\pa U^n\right).
\end{equation}
Using \eqref{fullyrec2}, \eqref{CNrelaxfull2} and notations \eqref{notation1}, \eqref{thetader}, \eqref{wder} we can prove the next propositions. We state them here without any proofs and we refer to \cite{KK} (Proposition~2.1 and Lemma~2.5) for further details (compare also with Propositions~\ref{proprec} and~\ref{diffprop} of this paper).
\begin{proposition}[properties of $\hat U$] 
The reconstruction $\hat U$ in \eqref{fullyrec1} satisfies $\hat U(t_n^+)=\R^n U^n$ and $\hat U(t_{n+1})=\R^{n+1}U^{n+1}$, $0\le n\le N-1$. In particular, $\hat U$ is continuous in time on $[0,T].$ Furthermore it satisfies
\begin{equation}
\label{recfullPDE}
\p_t\hat U-\ii\alpha\R^{n+1}\Theta=\ii\lambda\R^{n+1}\Pro^{n+1}\left(\Phi^{n+\frac12}U\right)+\frac{\R^{n+1}\Pro^{n+1}U^n-\R^nU^n}{k_n}\quad\text{ in } I_n. \qed
\end{equation}
\end{proposition}
\begin{proposition}[the difference $\hat U-\omega$]
The difference $\hat U-\omega$ can be expressed as 
\begin{equation}
\label{diffully}
\begin{aligned}
(\hat U-\omega)(t)&=-\frac12(t_{n+1}-t)(t-t_n)\R^{n+1}\pt W^{n+1}
\\&=-\frac\ii2(t_{n+1}-t)(t-t_n)\R^{n+1}\left(\alpha\pt\Theta^{n+1}+\lambda\Pro^{n+1}\left(\Phi^{n+\frac12}\pa U^n\right)\right)\quad t\in I_n. \qed
\end{aligned}
\end{equation}
\end{proposition}
%

\section{A posteriori error control for fully discrete schemes}\label{FDapost}
We are now ready to prove the main a posteriori error bound for the fully discrete scheme \eqref{CNrelaxfull}. 
We split the error $e:=u-U$ as 
$$e=u-U:=\hat\rho+\sigma+\epsilon\ \text{ with }\ \hat\rho:=u-\hat U,\ \sigma:=\hat U-\omega \ \text{ and }\ \epsilon:=\omega-U,$$
and we refer to $\hat\rho$ as the main error, to $\sigma$ as the time-reconstruction error, and to $\epsilon$ as the elliptic error.
%

\subsection{Estimation of $\sigma$ \& $\epsilon$}
In this subsection we give two simple propositions for the a posteriori estimation of the time-reconstruction and elliptic errors.
\begin{proposition}[estimation of $\sigma$]
For $0\le n\le N-1$, the following estimates hold:
\begin{equation}
\label{trestloc}
\max_{t_n\le t\le t_{n+1}}\|\sigma(t)\|\le\ep_{n+1}^{\TT,0}\quad\text{ with }\quad \ep_{n+1}^{\TT,0}:=\frac{k_{n}^2}8\left[\|\pt W^{n+1}\|+C_2\hh(\pt W^{n+1})\right]\quad \text{ and }
\end{equation}
%
%
\begin{equation}
\label{trestinftloc}
\max_{t_n\le t\le t_{n+1}}\|\sigma(t)\|_{L^\infty}\le\ep_{n+1}^{\TT,\infty}\ \text{ with }\ \ep_{n+1}^{\TT,\infty}:=\frac{k_{n}^2}8\left[\|\pt W^{n+1}\|_{L^\infty}+C_\infty\left(\ln\underline h_{n+1}\right)^2\hhin(\pt W^{n+1})\right].
\end{equation}
In particular, for $1\le m\le N,$ we have that
\begin{equation}
\label{trest}
\max_{0\le t\le t_m}\|\sigma(t)\|\le\E_m^{\TT,0}\quad \text{ with }\quad \E_m^{\TT,0}:=\max_{1\le n\le m}\ep_n^{\TT,0}\quad \text{ and }
\end{equation}
%
\begin{equation}
\label{trestinf}
\max_{0\le t\le t_m}\|\sigma(t)\|_{L^\infty}\le\E_m^{\TT,\infty}\ \text{ with }\ \E_m^{\TT,\infty}:=\max_{1\le n\le m}\ep_n^{\TT,\infty}.
\end{equation}
\end{proposition}
\begin{proof}
We write $\R^{n+1}\pt W^{n+1} = \pt W^{n+1} + (\R^{n+1} - \II)  \pt W^{n+1}$ and the estimates follow directly from \eqref{diffully} and the elliptic properties \eqref{l2ell1}, \eqref{linfell}.
\end{proof}
%
%
\begin{proposition}[estimation of $\epsilon$]
The next local estimates are valid, for $0\le n\le N-1,$
\begin{equation}
\label{ellipticestloc}
\max_{t_n\le t\le t_{n+1}}\|\epsilon(t)\|\le C_2\ep_{n+1}^{\Ss,0}\quad \text{ with }\quad \ep_{n+1}^{\Ss,0}:=\max\left\{\hh(U^{n+1}),\h(U^n)\right\}\quad \text{ and }
\end{equation}
\begin{equation}
\label{ellipticestinftloc}
\max_{t_n\le t\le t_{n+1}}\|\epsilon(t)\|_{L^\infty}\le C_\infty\ep_{n+1}^{\Ss,\infty}\ \text{ with }\ \ep_{n+1}^{\Ss,\infty}:=\max\left\{\left(\ln\underline h_{n+1}\right)^2\hhin(U^{n+1}),\left(\ln\underline h_n\right)^2\hin(U^n)\right\}.
\end{equation}
In particular, for $1\le m\le N$, the following global estimates hold true:
\begin{equation}
\label{ellipticest}
\max_{0\le t\le t_m}\|\epsilon(t)\|\le C_2\E_m^{\Ss,0}\quad \text{ with }\quad \E_m^{\Ss,0}:=\max_{0\le n\le m}\h(U^n)
\end{equation}
and
\begin{equation}
\label{ellipticestinft}
\max_{0\le t\le t_m}\|\epsilon(t)\|_{L^\infty}\le C_\infty\E_m^{\Ss,\infty}\quad \text{ with }\quad \E_m^{\Ss,\infty}:=\max_{0\le n\le m}\left(\ln\underline h_n\right)^2\hin(U^n). \quad\qed
\end{equation}
\end{proposition}
\begin{proof}
For $t\in I_n$, we write $\epsilon(t) = \ell_0^n(t)(\R^n-\II)U^n + \ell_1^{n+1}(\R^{n+1}-\II)U^{n+1}$ and the estimates follow from \eqref{l2ell1} and \eqref{linfell}.
\end{proof}
%
\subsection{Estimation of $\hat\rho$}
In this subsection we estimate a posteriori the main error $\hat\rho$ using energy techniques. The estimation of $\hat\rho$ is based on  the analysis of Section~\ref{apostimdis}.  We first have from \eqref{recfullPDE} that $\hat U$ satisfies, in $I_n$, $0\le n\le N-1$, the equation
\begin{equation}
\label{perteqful1}
\langle\p_t\hat U,v\rangle+\ii\alpha\langle\nabla\hat U,\nabla v\rangle=\ii\lambda\langle f(\hat U),v\rangle+\langle R,v\rangle,\quad\forall v\in H_0^1(\Om)
\end{equation}
with
$$R(t):=\R^{n+1}W(t)+\frac{\R^{n+1}\Pro^{n+1}U^n-\R^nU^n}{k_n}-\ii\lambda f\left(\hat U(t)\right)-\ii\alpha\D(\omega+\sigma)(t),\quad t\in I_n.$$
Furthermore, the definition of the elliptic reconstruction \eqref{ellrec} and \eqref{spacerec} lead to
$$\langle\D\omega,v\rangle=\ell_0^n(t)\langle\D^n U^n,v\rangle+\ell_1^n(t)\langle\D^{n+1}U^{n+1},v\rangle, \quad\forall v\in H_0^1(\Om)$$
and by \eqref{diffully}
$$\langle\D\sigma,v\rangle=-\frac12(t_{n+1}-t)(t-t_n)\langle\D^{n+1}\pt W^{n+1},v\rangle,\quad\forall v\in H_0^1(\Om).$$
Therefore, subtracting \eqref{perteqful1} from the first equation of \eqref{NLSvar} and using the above, leads to the next proposition:
\begin{proposition}[error equation for $\hat\rho$]
The main error $\hat\rho$ satisfies the following equation in $I_n$, $0\le n\le N-1$,
\begin{equation}
\label{errormain1}
\langle\p_t\hat\rho,v\rangle+\ii\alpha\langle\nabla\hat\rho,\nabla v\rangle=\ii\lambda\langle f(u)-f(\hat U),v\rangle+\sum_{j=1}^3\langle R_j,v\rangle,\quad\forall v\in H_0^1(\Om),
\end{equation}
where the residuals $R_j$, $1\le j\le 3$, are given by
\begin{equation}
\label{resful1}
R_1(t):=(\II-\R^{n+1})W(t)-\frac{\R^{n+1}\Pro^{n+1}U^n-\R^nU^n}{k_n}+\ii\alpha\ell_0^n(t)(\II-\Pro^{n+1})\D^nU^n,
\end{equation}
\begin{equation}
\label{resful2}
R_2(t):=-\frac{\ii\alpha}2(t_{n+1}-t)(t-t_n)\D^{n+1}\pt W^{n+1},
\end{equation}
and 
\begin{equation}
\label{resful3}
R_3(t):=\ii\lambda\left(f(\hat U)-\Pro^{n+1}\left(\Phi^{n+\frac12}U\right)\right)(t).\qquad \qed
\end{equation}
\end{proposition}
The next lemma is taken from \cite{KK} (Lemma~3.1) and we refer there for its proof.
\begin{lemma}[the residual $R_1$]
For $t\in I_n$, the residual $R_1$ can be rewritten equivalently as
\begin{equation}
\label{resful11}
\begin{aligned}
R_1(t)=(t-t_{n+\frac12})(\II-\R^{n+1})\pt W^{n+1}&-\frac{(\R^{n+1}-\II)U^{n+1}-(\R^n-\II)U^n}{k_n}\\&+(\II-\Pro^{n+1})\left(\ii\alpha\ell_0^n(t)\D^nU^n+\frac{U^n}{k_n}\right).\qquad \qed
\end{aligned}
\end{equation}
\end{lemma}

Equation \eqref{errormain1} is of the same form as \eqref{errorbasic}. Therefore, to estimate a posteriori $\hat\rho$ in the $L^\infty(L^2)-$norm we will follow similar arguments as for the time discrete schemes.  However note that in the case of fully discrete schemes the reconstruction $\hat U$ (as well as $\omega$) are no longer computable quantities. Thus before proceeding as in the time discrete case we need two auxiliary lemmata that will allow us to handle this issue.  

The first one gives a corresponding to \eqref{graderror4} estimate for $\|\hat\rho(t)\|^{pd}.$ In the proof of this lemma we use, for $d=2,$ the Cl\'{e}ment interpolant (\cite{BS,Clement,SZ})  or, for $d=1$, the piecewise linear interpolant $\widehat\I_{n+1} z\in\widehat{\V}^{n+1}:=\V^{n+1}\bigcap\V^n$ of $z\in H_0^1(\Om).$ We summarise next some of the properties of $\widehat\I_{n+1}z$ \cite{BS,Clement,SZ,PhDKyza}. The following estimates hold:
\begin{equation}
\label{Clem1}
\|\nabla(z-\widehat\I_{n+1} z)\|\le D_{2,0}\|\nabla z\|\quad \text{ and }
\end{equation}
\begin{equation}
\label{Clem2}
\sum_{K\in\widehat\T_n}\|h_K^{-1}(z-\widehat\I_{n+1} z)\|^2_{L^2(K)}\le D_{2,1}^2\|\nabla z\|^2,
\end{equation}
 where $D_{2,0}, D_{2,1}$ are absolute constants depending only on the shape regularity of the family of triangulations and on the number of bisections necessary to pass from $\T_n$ to $\T_{n+1}$.
 \begin{lemma}[estimation of $\|\nabla\hat\rho(t)\|^{pd}$]
 For $t\in I_n$, we define
 \begin{equation}
 \label{Mu}
 M(t):=\D^n U^n+\frac{t-t_n}{k_n}\left(\D^{n+1}\Pro^{n+1}U^n-\D^nU^n\right)+\frac{t-t_n}2\D^{n+1}\left(W(t)+W(t_n)\right)\quad \text{ and }
 \end{equation}
 \begin{equation}
 \label{Zet}
 Z(t):= U^n+\frac{t-t_n}{k_n}\left(\Pro^{n+1}U^n-U^n\right)+\frac{t-t_n}2\left(W(t)+W(t_n)\right).
 \end{equation}
 Then the following estimate holds for $t\in I_n,\, 0\le n\le N-1,$
 \begin{equation}
 \label{gradestpr}
 \|\nabla\hat\rho(t)\|\le\|\nabla u(t)\|+\max\{D_{2,0},D_{2,1}\}\left(\|\nabla Z(t)\|+\|\widehat{h}_{n+1}M(t)\|\right).
 \end{equation}
 In particular, due to \eqref{SI} we have that
 \begin{equation}
 \label{gradestmain}
 \|\nabla\hat\rho(t)\|^{pd}\le\gamma(pd)\left(\|\nabla u(t)\|^{pd}+\left[\max\{D_{2,0},D_{2,1}\}\left(\|\nabla Z(t)\|+\|\widehat{h}_{n+1}M(t)\|\right)\right]^{pd}\right).
 \end{equation}
 \end{lemma}
\begin{proof}
First we write
\begin{equation}
\label{gradest1}
\|\nabla\hat\rho(t)\|^2=\langle\nabla u,\nabla\hat\rho\rangle(t)-\langle\nabla \hat U,\nabla\hat\rho\rangle(t),
\end{equation}
and using the definition of the elliptic reconstruction \eqref{ellrec} and \eqref{fullyrec2}, we deduce
$$-\langle\nabla\hat U,\nabla\hat\rho\rangle(t)=\langle M,\hat\rho\rangle(t)\le\|M(t)\|_{H^{-1}}\|\nabla\hat\rho(t)\|.$$
Therefore, \eqref{gradest1} takes the form
\begin{equation}
\label{gradestfull2}
\|\nabla\hat\rho(t)\|^2\le\left(\|\nabla u(t)\|+\|M(t)\|_{H^{-1}}\right)\|\nabla\hat\rho(t)\|,\quad \text{ or }\quad \|\nabla\hat\rho(t)\|\le\|\nabla u(t)\|+\|M(t)\|_{H^{-1}}.
\end{equation}
For the estimation of $\|M(t)\|_{H^{-1}}$ we use the definition of the discrete laplacian \eqref{discrlapl} to obtain
\begin{equation}
\label{Mu2}
\begin{aligned}
\|M(t)\|_{H^{-1}}&=\sup_{\substack{z\in H_0^1(\Om) \\ z\ne 0}}\frac{\langle M(t),z\rangle}{\|\nabla z\|} =\sup_{\substack{z\in H_0^1(\Om) \\ z\ne 0}}\left\{\frac{\langle M(t),\widehat\I_{n+1}z\rangle}{\|\nabla z\|}+\frac{\langle M(t),z-\widehat\I_{n+1}z\rangle}{\|\nabla z\|}\right\},\\
&=\sup_{\substack{z\in H_0^1(\Om) \\ z\ne 0}}\left\{\frac{\langle \nabla Z(t),\nabla\widehat\I_{n+1}z\rangle}{\|\nabla z\|}+\frac{\langle \widehat h_{n+1}M(t),\widehat h_{n+1}^{-1}(z-\widehat\I_{n+1}z)\rangle}{\|\nabla z\|}\right\}
\end{aligned}
\end{equation}
where, recall that $\widehat\I_{n+1}z\in\widehat\V^{n+1}$ denotes, for $d=2,$ the Cl\'ement interpolant and, for $d=1$, the piecewise linear interpolant of $z\in H_0^1(\Om)$. The desirable result \eqref{gradestpr} is then obtained by applying \eqref{Clem1}, \eqref{Clem2} in the last equality of \eqref{Mu2} and inserting the result back to \eqref{gradestfull2}.
\end{proof}

Next we would like to deduce an estimate for $R_3$. For this we write $R_3:=\ii\lambda(R_{31}+R_{32}+R_{33})$ with
\begin{equation}
\label{parR4}
 R_{31}:=f(\hat U)-f(\omega),\ R_{32}:=f(\omega)-f(U),\ \text{ and } R_{33}:=f(U)-\Pro^{n+1}\left(\Phi^{n+\frac12}U\right).
\end{equation}
The term $R_{33}$ is already an a posteriori quantity, while we estimate $R_{31}$ and $R_{32}$ in the next lemma.
\begin{lemma}[estimation of $R_{31}$ and $R_{32}$]
For $t\in I_n$, $0\le n\le N-1$, we have that 
\begin{equation}
\begin{aligned}
\label{estR41}
\|R_{31}(t)\|\le\Les_{31}^{n+1}(t_{n+1}-t)(t-t_n)\quad \text{ with }
\end{aligned}
\end{equation}
\begin{equation}
\label{upestR41}
\Les_{31}^{n+1}:=(p+\frac12)\left(\ep_{n+1}^{\TT,\infty}+\ep_{n+1}^{\Ss,\infty}+\max\left\{\|U^n\|_{L^\infty},\|U^{n+1}\|_{L^\infty}\right\}\right)^{2p}\left(\|\pt W^{n+1}\|+C_2\hh(\pt W^{n+1})\right)
\end{equation}
and
\begin{equation}
\label{estR42}
\|R_{32}(t)\|\le C_2\Les_{32}^{n+1}\max\left\{\hh(U^{n+1}),\h(U^n)\right\}\quad \text{ with }
\end{equation}
\begin{equation}
\label{upestR42}
\Les_{32}^{n+1}:=(2p+1)\left(C_{\infty}\ep_{n+1}^{\Ss,\infty}+\max\left\{\|U^n\|_{L^\infty},\|U^{n+1}\|_{L^\infty}\right\}\right)^{2p},
\end{equation}
where $C_2,C_{\infty}$ are the constants in \eqref{resest1} and \eqref{resest3}, respectively, and $\ep_{n+1}^{\TT,\infty}$, $\ep_{n+1}^{\Ss,\infty}$ are the estimators in \eqref{trestinftloc} and \eqref{ellipticestinftloc}, respectively.
\end{lemma}
\begin{proof}
Using \eqref{NLdiff} with $z_1=\hat U$ and $z_2=\omega$ and $v=\sigma$ we estimate
$$|R_{31}|\le (2p+1)|s\hat U+(1-s)\omega|^{2p}|\sigma|=(2p+1)|s\sigma+\epsilon+U|^{2p}|\sigma|,$$
where recall that $s=s(t)\in[0,1].$ Furthermore, since $2p\ge 1,$ we have that
$$\|R_{31}(t)\|\le (2p+1)\left(\|\sigma(t)\|_{L^\infty}+\|\epsilon\|_{L^\infty}+\|U(t)\|_{L^\infty}\right)^{2p}\|\sigma(t)\|,\quad t\in I_n.$$
Hence, estimate \eqref{estR41} is then obtained using \eqref{trestinftloc}, \eqref{ellipticestinftloc}, the definition of $U$ and \eqref{diffully}. Similarly we deduce estimate \eqref{estR42}, but now instead of using \eqref{trestinftloc} and \eqref{diffully}, we use \eqref{ellipticestloc}. 
\end{proof}

We are now ready to state and prove the main theorem of the subsection.
\begin{theorem}[estimation of $\hat\rho$]\label{rhoth}
For $t\in I_n$, let
\begin{equation*}
\begin{aligned}
&\mathcal{M}(U,u_0;t):=A\Bigg(\gamma(pd)\left(\|u_0\|^{p(2-d)}+\left(\ep_{n+1}^{\TT,0}+C_2\ep_{n+1}^{\Ss,0}+\|U(t)\|\right)^{p(2-d)}\right)\\
&\times\left(G(u_0)^{pd}+\left[\max\{D_{2,0},D_{2,1}\}\left(\|\nabla Z(t)\|+\|\widehat{h}_{n+1}M(t)\|\right)\right]^{pd}\right)+\left(\ep_{n+1}^{\TT,\infty}+C_{\infty}\ep_{n+1}^{\Ss,\infty}+\|U(t)\|_{L^\infty}\right)^{2p}\Bigg),
\end{aligned}
\end{equation*}
where $A$ is the constant in \eqref{nlp3},    $\gamma(pd)$ is as in \eqref{SI}, $C_2, C_{\infty}$ are the constants in \eqref{resest1}, \eqref{resest2}, $D_{2,0},D_{2,1}$ are the constants in \eqref{Clem1}, \eqref{Clem2}, $\ep_{n+1}^{\TT,0}, \ \ep_{n+1}^{\TT,\infty},\ \ep_{n+1}^{\Ss,0},\ \ep_{n+1}^{\Ss,\infty}$ are as in \eqref{trestloc}, \eqref{trestinftloc}, \eqref{ellipticestloc}, \eqref{ellipticestinftloc}, respectively. $M(t),Z(t)$ are given by \eqref{Mu}, \eqref{Zet}, $G(u_0)$ is given by \eqref{boundgrad}. 

Then, for $0\le n\le N-1$, the following local estimate holds true for the main error $\hat\rho:=u-\hat U$:
\begin{equation}
\label{rholocal}
\begin{aligned}
\sup_{t\in I_n}\|\hat\rho(t)\|\le&\exp\left(\int_{I_n}\mathcal{M}(U,u_0;t)\,dt\right)\\&\times\left(\|\hat\rho(t_n)\|+\ep_{n+1}^{\TT,1}+\ep_{n+1}^{\TT,2}+C_2(\ep_{n+1}^{\Ss,1}+\ep_{n+1}^{\Ss,2})+\widehat C_2\ep_{n+1}^{\Ss,3}+\ep_{n+1}^{\mathrm{C}}+\ep_{n+1}^{\mathrm{D}}\right)
\end{aligned}
\end{equation}
with
$$\ep_{n+1}^{\TT,1}:=\frac{\alpha k_n^3}{12}\|\D^{n+1}\pt W^{n+1}\|\quad \text{ and }\quad \ep_{n+1}^{\TT,2}:=\frac{k_n^3}6\Les_{31}^{n+1},$$
\begin{equation*}
\begin{aligned}
\ep_{n+1}^{\Ss,1}:=\frac{k_n^2}4\hh(\pt W^{n+1}), &\quad \ep_{n+1}^{\Ss,2}:=\Les_{32}^{n+1}\max\left\{\hh(U^{n+1}),\h(U^n)\right\},\\&\text{and}\quad \ep_{n+1}^{\Ss,3}:=k_n\het(\frac{U_{n+1}}{k_n},\frac{U_n}{k_n}), 
\end{aligned}
\end{equation*}
$$\ep_{n+1}^{\mathrm{C}}:=\int_{I_n}\|(\II-\Pro^{n+1})\left(\frac{U^n}{k_n}+\ii\alpha\ell_0^n(t)\D^nU^n\right)\|\,dt,$$
$$\text{and}\quad \ep_{n+1}^{\mathrm{D}}:=\int_{I_n}\|\left(f(U)-\Pro^{n+1}\left(\Phi^{n+\frac12}U\right)\right)(t)\|\,dt,
$$
with $\widehat C_2$ the constant in \eqref{resest3},  $\Les_{31}^{n+1}$ and $\Les_{32}^{n+1}$ as in \eqref{upestR41} and \eqref{upestR42}, respectively, and where $\|\hat\rho(0)\|$ is estimated as $\|\hat\rho(0)\|\le \|u_0-U^0\|+C_2\eta_{2,\V^0}(U^0).$ In particular, for $1\le m\le N,$ the following global estimate is valid for $\hat\rho$:
\begin{equation}
\label{rhoglobal}
\begin{aligned}
\max_{0\le t\le t_m}\|\hat\rho(t)\|\le\exp\left(\sum_{n=0}^{m-1}\int_{I_n}\mathcal{M}(U,u_0;t)\,dt\right)\times\Big(&\|u_0-U^0\|+C_2\eta_{2,\V^0}(U^0)\\&\hspace{-2.5cm}+\E_m^{\TT,1}+\E_m^{\TT,2}+C_2
(\E_m^{\Ss,1}+\E_m^{\Ss,2})+\widehat C_2\E_m^{\Ss,3}+\E_m^{\mathrm{C}}+\E_m^{\mathrm{D}}\Big),
\end{aligned}
\end{equation}
with $\E_m^{\TT,j}:=\sum_{n=1}^m\ep_n^{\TT,j}$, $j=1,2$, $\E_m^{\Ss,j}:=\sum_{n=1}^m\ep_n^{\Ss,j}$, $j=1,2,3,$ $\E_m^{\mathrm{C}}:=\sum_{n=1}^m\ep_n^{\mathrm{C}},$ and $\E_m^{\mathrm{D}}:=\sum_{n=1}^m\ep_n^{\mathrm{D}}.$
\end{theorem}
\begin{proof}
Equation \eqref{errormain1} is of the same form as \eqref{errorbasic}. Therefore, using \eqref{gradestmain} instead of \eqref{graderror4} and proceeding similarly to the proof of  Theorem~\ref{apostdiscth} leads to
\begin{equation}
\label{rho1}
\sup_{t\in I_n}\|\hat\rho(t)\|\le\exp\left(\int_{I_n}\mathcal{H}(\hat U,u_0;t)\,dt\right)\times\left(\|\hat\rho(t_n)\|+\sum_{j=1}^3\int_{I_n}\|R_j(t)\|\,dt\right),\quad 0\le n\le N-1,
\end{equation}
where $\mathcal{H}(\hat U,u_0;t)$ is as in \eqref{recfunctional}, but with $\|\nabla\hat U(t)\|$ being replaced by $\max\{D_{2,0},D_{2,1}\}\Big(\|\nabla Z(t)\|+$ $\|\widehat{h}_{n+1}M(t)\|\Big)$ (cf., \eqref{gradestmain}). Moreover, from \eqref{trestloc} and \eqref{ellipticestloc} we get
$$\|\hat U(t)\|\le\|\sigma(t)\|+\|\epsilon(t)\|+\|U(t)\|\le\ep_{n+1}^{\TT,0}+C_2\ep_{n+1}^{\Ss,0}+\|U(t)\|,$$
while similarly, from \eqref{trestinftloc} and \eqref{ellipticestinftloc} we deduce
$$\|\hat U(t)\|\le\ep_{n+1}^{\TT,\infty}+C_\infty\ep_{n+1}^{\Ss,\infty}+\|U(t)\|.$$
Hence \eqref{rho1} becomes
\begin{equation}
\label{rho2}
\sup_{t\in I_n}\|\hat\rho(t)\|\le\exp\left(\int_{I_n}\mathcal{M}( U,u_0;t)\,dt\right)\times\left(\|\hat\rho(t_n)\|+\sum_{j=1}^3\int_{I_n}\|R_j(t)\|\,dt\right),\quad 0\le n\le N-1.
\end{equation}

On the other hand, due to \eqref{resful11}, \eqref{resest1} and \eqref{resest3}, it is easily seen that
\begin{equation}
\label{estres1fin}
\int_{I_n}\|R_1(t)\|\,dt\le C_2\ep_{n+1}^{\Ss,1}+\widehat{C}_2\ep_{n+1}^{\Ss,3}+\ep_{n+1}^{\mathrm{C}}.
\end{equation}
Additionally, from \eqref{resful2}, \eqref{resful3}, \eqref{parR4}, \eqref{estR41}, and \eqref{estR42}, we deduce
\begin{equation}
\label{estres23fin}
\int_{I_n}\|R_2(t)\|\,dt\le\ep_{n+1}^{\TT,1}\quad \text{and}\quad \int_{I_n}\|R_3(t)\|\,dt\le\ep_{n+1}^{\TT,2}+C_2\ep_{n+1}^{\Ss,2}+\ep_{n+1}^{\mathrm{D}}.
\end{equation}

Plugging \eqref{estres1fin} and \eqref{estres23fin} into \eqref{rho2} we readily obtain \eqref{rholocal}. Estimate \eqref{rhoglobal} follows from \eqref{rholocal} by summing over $n.$
\end{proof}

Similarly to Theorem~\ref{impth}, we can also derive an improved estimate for $\|\hat\rho(t)\|$ in the one-dimensional case.
\begin{theorem}[improved estimate for $\hat\rho$ in $d=1$]\label{impth1d}
With the notation of Theorem~\ref{rhoth}  we further denote by $\eta(t)$ the right-hand side in \eqref{rhoglobal} to the power $p$. Let
\begin{equation}
\label{impth1dexp}
\begin{aligned}
\mathcal{N}(U,u_0;t):=A\Bigg(&\eta(t)\left(G(u_0)^{p}+\left[\max\{D_{2,0},D_{2,1}\}\left(\|\nabla Z(t)\|+\|\widehat{h}_{n+1}M(t)\|\right)\right]^{p}\right)\\&
+\left(\ep_{n+1}^{\TT,\infty}+C_{\infty}\ep_{n+1}^{\Ss,\infty}+\|U(t)\|_{L^\infty}\right)^{2p}\Bigg).
\end{aligned}
\end{equation}
Then, if $d=1,$ the following local and global estimates are valid:
\begin{equation}
\label{rho1dlocal}
\begin{aligned}
\sup_{t\in I_n}\|\hat\rho(t)\|\le&\exp\left(\int_{I_n}\mathcal{N}(U,u_0;t)\,dt\right)\times\Big(\|\hat\rho(t_n)\|+\ep_{n+1}^{\TT,1}+\ep_{n+1}^{\TT,2}\\&+C_2(\ep_{n+1}^{\Ss,1}+\ep_{n+1}^{\Ss,2})+\widehat C_2\ep_{n+1}^{\Ss,3}+\ep_{n+1}^{\mathrm{C}}+\ep_{n+1}^{\mathrm{D}}\Big),\quad 0\le n\le N-1,
\end{aligned}
\end{equation}
and
\begin{equation}
\label{rho1dglobal}
\begin{aligned}
\max_{0\le t\le t_m}&\|\hat\rho(t)\|\le\exp\left(\sum_{n=0}^{m-1}\int_{I_n}\mathcal{N}(U,u_0;t)\,dt\right)\times\Big(\|u_0-U^0\|+C_2\eta_{2,\V^0}(U^0)\\&+\E_m^{\TT,1}+\E_m^{\TT,2}+C_2
(\E_m^{\Ss,1}+\E_m^{\Ss,2})+\widehat C_2\E_m^{\Ss,3}+\E_m^{\mathrm{C}}+\E_m^{\mathrm{D}}\Big),\quad 1\le m\le N.\quad \qed
\end{aligned}
\end{equation}
\end{theorem}
%
\subsection{A posteriori error bounds in the $L^\infty(L^2)-$norm for fully discrete schemes}
We conclude the section by presenting the theorems with the a posteriori error estimates in the $L^\infty(L^2)-$norm for the error $e=u-U$.
\begin{theorem}[a posteriori error estimate in the $L^\infty(L^2)$]\label{fulaposth}
With the notation of Theorem~\ref{rhoth} we further define
\begin{equation*}
\begin{aligned}
\ep_{n+1}^{\mathrm{LOC}}:=\exp\left(\int_{I_n}\mathcal{M}(U,u_0;t)\,dt\right)\times\Big(\ep_n^{\mathrm{LOC}}
+\ep_{n+1}^{\TT,1}+\ep_{n+1}^{\TT,2}&+C_2(\ep_{n+1}^{\Ss,1}+\ep_{n+1}^{\Ss,2})+\widehat C_2\ep_{n+1}^{\Ss,3}\\&+\ep_{n+1}^{\mathrm{C}}+\ep_{n+1}^{\mathrm{D}}\Big), \quad 0\le n\le N-1,
\end{aligned}
\end{equation*}
with $\ep_0^{\mathrm{LOC}}:=\|u_0-U^0\|+C_2\eta_{2,\V^0}(U^0).$ Let also $\E_m^{\mathrm{GLOB}}$, $1\le m\le N$, be the right-hand side in \eqref{rhoglobal}. Then the following estimates are valid:
\begin{equation}
\label{fullapostloc}
\sup_{t\in I_n}\|(u-U)(t)\|\le C_2\ep_{n+1}^{\Ss,0}+\ep_{n+1}^{\TT,0}+\ep_{n+1}^{\mathrm{LOC}},\quad 0\le n\le N-1,
\end{equation}
and 
\begin{equation}
\label{fullapostglob}
\max_{0\le t\le t_m}\|(u-U)(t)\|\le C_2\E_m^{\Ss,0}+\E_m^{\TT,0}+\E_m^{\mathrm{GLOB}},\quad 1\le m\le N,
\end{equation}
where $u$ is the exact solution of \eqref{NLS}, $U$ denotes the continuous, linear, piecewise interpolant $U(t)=\ell_0^n(t)U^n+\ell_1^n(t)U^{n+1},\, t\in I_n$, and $\{U^n\}_{n=0}^N$ are the modified relaxation Crank-Nicolson-Galerkin approximations, obtained by \eqref{CNrelaxfull}. 
\end{theorem}
\begin{proof}
We just note that
$$\|(u-U)(t)\|=\|e(t)\|\le\|\epsilon(t)\|+\|\sigma(t)\|+\|\hat\rho(t)\|$$
and we apply \eqref{ellipticest}, \eqref{trestloc}, and \eqref{rholocal} to obtain \eqref{fullapostloc}, while for \eqref{fullapostglob} we use \eqref{ellipticest}, \eqref{trest} and \eqref{rhoglobal}.
\end{proof}
The last theorem of the section is the improved a posteriori error estimate for $u-U$ in the $L^\infty(L^2)-$norm for the one-dimensional case.
\begin{theorem}[improved a posteriori error estimate in the $L^\infty(L^2)$ for $d=1$]
With the notation of Theorems~\ref{impth} and~\ref{fulaposth} we further define $\tilde{\ep}_n^{\mathrm{LOC}},\, 0\le n\le N$, and $\tilde{\E}_m^{\mathrm{GLOB}},\, 1\le m\le N$, as in Theorem~\ref{fulaposth}, but with $\mathcal{M}$ replaced by $\mathcal{N}.$ Then, for $d=1$, the following local and global estimates hold:
\begin{equation*}
\sup_{t\in I_n}\|(u-U)(t)\|\le C_2\ep_{n+1}^{\Ss,0}+\ep_{n+1}^{\TT,0}+\tilde\ep_{n+1}^{\mathrm{LOC}},\quad 0\le n\le N-1,
\end{equation*}
and
\begin{equation}
\label{fullapost1dglob}
\max_{0\le t\le t_m}\|(u-U)(t)\|\le C_2\E_m^{\Ss,0}+\E_m^{\TT,0}+\tilde\E_m^{\mathrm{GLOB}},\quad 1\le m\le N.\quad \qed
\end{equation}
\end{theorem}
%
\section{Numerical Experiments}\label{numexp}

In this section we report a series of numerical experiments which verify the theoretical results obtained in the previous sections. The modified relaxation Crank-Nicolson-Galerkin scheme \eqref{CNrelaxfull} and the corresponding a posteriori space and time estimators  with homogeneous Dirichlet boundary conditions were implemented in a double precision C-code using B-splines of degree $r, \ r\in\Nb$, as a basis for the finite element space $\V^n, \ 0\le n \le N$. The numerical results are for the one-dimensional case $d=1$; in this case $\varOmega=[a,b].$ For the implementations we use uniform partitions in space and time and we set equal to one the absolute constants $A,\beta, C_2, \widehat{C}_2, C_\infty, D_{2,0}, D_{2,1}$ appearing in the final a posteriori error estimators.

Our goals in this section are :  (a) to asses the magnitude of the constants $\Les_{31}^n, \ \Les_{32}^n$, $\M(U,u_0;t)$ and $\N(U,u_0;t)$ involved in the estimators and (b) to verify the correct  order of convergence for the space and time estimators.

To facilitate the process we choose to work with a soliton type exact solution of \eqref{NLS}. In particular we choose, \cite{Besse}
\begin{equation}
\label{esol}
u(x,t) = \ii\,\text{sech}(x-x_0-4\omega t)\,\text{e}^{\ii \left( 2\omega x +(1-4\omega^2)t   \right)},
\end{equation}
which is an exact solution of \eqref{NLS} with  $p=1, \ \alpha =1 , \ \lambda = 2$ and $u_0=u(\cdot,0).$  The function $u$ represents a solitary wave,  initially located around $x=x_0$, which travels to the right with speed $4\omega$.  All computations are performed in $[a,b] \times [0,T]= [-30,30]\times [0,1]$ with $x_0=0$ and $\omega=0.3$.

\begin{remark}
\upshape
The development of an adaptive space time algorithm based on the a posteriori error  estimators of the paper and its numerical validation in terms of accuracy, effectiveness and robustness for the approximation of the  solution $u$  of  \eqref{NLS} will be the subject of a forthcoming paper. In this forthcoming work we will also study further cases including for example nonlinearities to the critical exponent $p^*=\frac2d$, as well as particular examples for \eqref{NLS} in the semiclassical regime, cf. \cite{TVZ}. 
\end{remark}

In the next subsections we consider a series of different runs for the NLS problem  \eqref{NLS} with exact solution \eqref{esol}.  Let  $\ell\in\Nb$ counts the different realizations (runs), $h(\ell)$ the corresponding meshsize and $M(\ell)=1+\big[\frac{b-a}{h(\ell)}\big]\in\Nb$ where $[\cdot]$ denotes the integral part of a real number. When there is no danger of confusion we drop the dependence on $\ell$ and we write just $h$ and $M$. Since we discretize in space by B-splines of degree $r$, the expected order of convergence for the modified relaxation Crank-Nicolson finite element scheme is $r+1$ in space, whilst in time is $2$. This motivates the relation between the mesh size $h$ and the time step $k$. More specifically  in all the computations of the next two subsections we choose
\begin{equation}
\label{khrel}
k\sim h^{\frac{r+1}2}.
\end{equation}
\subsection{Behavior of  $\Les_{31}^n$, $\Les_{32}^n$ and $\M, \N$}
First, we quantify  the quantities  $\Les_{31}^n$, $\Les_{32}^n$ defined in \eqref{upestR41}, \eqref{upestR42}, respectively. For $0\le n\le N-1,$ $\Les_{31}^n$ appears in the local time estimator $\ep_{n+1}^{\TT,2}$, cf. \eqref{rholocal}, and thus in the global time estimator $\E_N^{\TT,2}$, cf. \eqref{rhoglobal}. Similarly, for $0\le n\le N-1,$ $\Les_{32}^n$ is in the local space estimator  $\ep_{n+1}^{\Ss,2}$, cf. \eqref{rholocal}, and thus in the global space estimator $\E_N^{\Ss,2}$, cf. \eqref{rhoglobal}. Estimators $\E_N^{\TT,2}$ and $\E_N^{\Ss,2}$ are expected to be of optimal order of accuracy as long as the following global quantities  
$$\Les_{31}:=\max_{0\le n\le N-1}\Les_{31}^n\quad \text{ and }\quad  \Les_{32}^n:=\max_{0\le n\le N-1}\Les_{31}^n,$$
tend to a constant value for different realizations. In Table~\ref{expN} we compute the values of $\Les_{31}, \Les_{32}$ for linear ($r=1$), quadratic ($r=2$) and cubic ($r=3$) B-splines and a series of different space discretizations $h$; the time step is chosen according to \eqref{khrel}.

We are also interested in studying the magnitude of the quantity $\displaystyle \E_N^{\M} := \sum_{n=0}^{N-1} \int_{I_n}\M(U,u_0;t)\,dt$, which appears in the exponential  of \eqref{rhoglobal} and thus on the final global a posteriori error estimator \eqref{fullapostglob}. Our findings are reported again in Table~\ref{expN}. Since the considered example is in the one spatial dimension, the improved a posteriori error estimate \eqref{fullapost1dglob} holds (see also \eqref{rho1dglobal}). The difference  of this estimate compared to \eqref{fullapostglob} is the existence of $\displaystyle \E_N^{\mathcal{N}} := \sum_{n=0}^{N-1} \int_{I_n}\mathcal{N}(U,u_0;t)\,dt$ instead of    $\displaystyle \E_N^{\M} = \sum_{n=0}^{N-1} \int_{I_n}\M(U,u_0;t)\,dt$ in the exponential term. In Table~\ref{expN} we confirm that indeed $\E_N^{\mathcal{N}}$ is smaller than $\E_N^{\M}$. 
More precisely, in view of \eqref{impth1dexp}, we expect that, as the temporal and spatial mesh sizes tend to zero,  $\E_N^{\mathcal{N}}$ will tend  to $\displaystyle\int_0^T\|u(t)\|^{2p}_{L^\infty}\,dt$, which is equal to $1$ for this particular example. This is verified in Table~\ref{expN} for quadratic and cubic B-splines. For linear B-splines, we need to consider smaller mesh sizes; however, it is clear that $\E_N^{\mathcal{N}}$ reduces as $h$ reduces in this case as well.  Although at first glance the improvement on $\E_N^{\mathcal{N}}$ compared to $\E_N^{\M}$  seems minor, we should keep in mind that these quantities appear in the exponential, therefore even minor improvements have a significant impact on the size of the total a posteriori error estimator.  

\begin{table}[ht]
\begin{center}
\renewcommand{\arraystretch}{1.1}
\begin{tabular}{|c||c|c|c|c||c|c|c|c||c|c|c|c|}\hline
 & \multicolumn{4}{c||}{$r=1$} & \multicolumn{4}{c||}{$r=2$} & \multicolumn{4}{c|}{$r=3$} \\ \hline\hline
$M$ & $\Les_{31}$ & $\Les_{32}$ &  $\E_N^{\M}$ &  $\E_N^{\N}$ & $\Les_{31}$ & $\Les_{32}$ &  $ \E_N^{\M}$ &  $\E_N^{\N}$ & $\Les_{31}$ & $\Les_{32}$ &  $ \E_N^{\M}$ &  $\E_N^{\N}$\\ \hline
$2400$ & $6.639$ & $3.070$ & $7.847$ & $17.423$& $6.359$ & $3.001$ & $7.820$ &$ 2.233$ &$ 6.321 $& $3.000$ & $7.820$ & $1.164 $\\
$3600$ & $6.520$ & $3.038$ & $7.831$ & $9.757  $& $6.338$ & $3.000$ & $7.817$ &$ 1.613$ &$ 6.317 $& $3.000$ & $7.817$ & $1.071 $\\
$4800$ & $6.464$ & $3.025$ & $7.825$ & $6.737  $& $6.329$ & $3.000$ & $7.816$ &$ 1.382$ &$ 6.315 $& $3.000$ & $7.816$ & $1.040 $\\
$6000$ & $6.432$ & $3.017$ & $7.821$ & $5.187  $& $6.325$ & $3.000$ & $7.815$ &$ 1.267$ &$ 6.315 $& $3.000$ & $7.815$ & $1.025 $\\
$7200$ & $6.411$ & $3.013$ & $7.820$ & $4.263  $& $6.322$ & $3.000$ & $7.814$ &$ 1.201$ &$ 6.314 $& $3.000$ & $7.815$ & $1.018 $\\
$8400$ & $6.397$ & $3.010$ & $7.818$ & $3.658  $& $6.320$ & $3.000$ & $7.813$ &$ 1.158$ &$ 6.314 $& $3.000$ & $7.815$ & $1.013 $\\
$9600$ & $6.386$ & $3.008$ & $7.818$ & $3.234  $& $6.319$ & $3.000$ & $7.813$ &$ 1.128$ &$ 6.314 $& $3.000$ & $7.815$ & $1.010 $\\
\hline
\end{tabular}\\[2ex]
\caption{The computed quantities  $\Les_{31}$, $\Les_{32}$,  $\E_N^{\M}$ and $\E_N^{\N}$.} \label{expN}
\end{center}
\end{table}

\subsection{EOC of the estimators}
Our purpose in this subsection is to compute the \emph{experimental order of convergence (EOC)} of the space and time estimators at the final time $T=1$ for the NLS problem \eqref{NLS} with exact solution \eqref{esol}. For this, we choose quadratic B-splines ($r=2$). Hence the expected order of convergent of the modified relaxation Crank-Nicolson finite element scheme is in this case $3$ in space and $2$ in time; thus we take $k\sim  h^{3/2},$ cf. \eqref{khrel}.

The EOC for each space estimator $\E_{N}^{S,j}, \ 0\le j \le 3,$ is computed as follows 
\begin{equation}
\label{seoc}
\text{EOC}_{\Ss,j}:=\frac{\log(\E_{N}^{\Ss,j}(\ell)/\E_{N}^{\Ss,j}(\ell+1))} {\log(M({\ell+1})/M({\ell}))},
\end{equation}
where $\E_{N}^{\Ss,j}(\ell)$ and $\E_{N}^{\Ss,j}(\ell+1)$ denote the value of the estimators in two consecutive implementations with mesh sizes $h(\ell)$ and $h(\ell+1)$ respectively. Similarly for the time estimators $\E_{N}^{\TT,j}, \ 0\le j \le 2$, the EOC is computed as 
\begin{equation}
\label{teoc}
\text{EOC}_{\TT,j}:=\frac{\log(\E_{N}^{\TT,j}(\ell)/\E_{N}^{\TT,j}(\ell+1))} {\log(k({\ell+1})/k({\ell}))},
\end{equation}
where $k(\ell), \ k(\ell+1)$ are the time steps of consecutive realizations.  

In Table~\ref{tspace} the values of the space estimators $\E_N^{\Ss,j}$ along with the corresponding $\text{EOC}_{\Ss,j},\, 0\le j\le 3,$  are presented. Note that the estimator $\E_N^{\Ss,1}$ is expected to be of optimal third order in space and of first order in time (cf. Theorem~\ref{rhoth}), i.e., it is a superconvergent term. Due to the choice $k\sim  h^{3/2}$ we have $k\times h^3=h^{3/2}\times h^3=h^{9/2}$, and the EOC we expect to observe is $4.5$. 
This EOC is indeed observed in Table~\ref{tspace} for $\E_N^{\Ss,1}$, while, as it is shown in the same table, the other three space estimators exhibit the correct order.
Similarly, in Table~\ref{ttime} the values of the time estimators $\E_N^{\TT,j}$ along with  the  $\text{EOC}_{\TT,j},\, 0\le j\le 2,$  are presented. We note that all time estimators exhibit the correct second order. 
%
 %
\begin{table}[!ht]
\begin{center}
\renewcommand{\arraystretch}{1.1}
\begin{tabular}{|c||cc|cc|cc|cc|}\hline
$M$ &  $ \E_N^{\Ss,0}$ & $\text{EOC}_{\Ss,0}$&$\E_N^{\Ss,1}$ & $\text{EOC}_{\Ss,1}$& $\E_N^{\Ss,2}$& $\text{EOC}_{\Ss,2}$&$\E_N^{\Ss,3}$ &$\text{EOC}_{\Ss,3}$ \\ \hline\hline
$2400$ &$ 1.3817$E$-05$ & --          & $2.64149$E$-07$ &   --       & $4.1466$E$-05 $&   --        &$ 2.5055$E$-05 $& -- \\
$3600$ &$ 4.0936$E$-06$ &$ 3.0001$ & $4.25720$E$-08$ & $4.5018 $& $1.2283$E$-05 $& $3.0007 $&$ 7.4230$E$-06 $&$ 3.0003 $\\
$4800$ &$ 1.7270$E$-06$ &$ 3.0000$ & $1.16949$E$-08$ & $4.4912 $& $5.1813$E$-06 $& $3.0003 $&$ 3.1315$E$-06 $&$ 3.0001 $\\
$6000$ &$ 8.8421$E$-07$ &$ 3.0000$ & $4.28570$E$-09$ & $4.4988 $& $2.6527$E$-06 $& $3.0002 $&$ 1.6033$E$-06 $&$ 3.0001 $\\
$7200$ &$ 5.1169$E$-07$ &$ 3.0000 $& $1.88822$E$-09 $& $4.4956 $& $1.5351$E$-06 $& $3.0001 $&$ 9.2783$E$-07 $&$ 3.0000 $\\
$8400$ &$ 3.2223$E$-07$ &$ 3.0000$ & $9.44191$E$-10$ & $4.4960 $& $9.6671$E$-07 $& $3.0001 $&$ 5.8429$E$-07 $&$ 3.0000 $\\
$9600$ &$ 2.1587$E$-07$ &$ 3.0000$ & $5.20390$E$-10$ & $4.4615 $& $6.4762$E$-07 $& $3.0000 $&$ 3.9142$E$-07 $&$ 3.0000 $\\
\hline
\end{tabular}\\[2ex]
\caption{Space a posteriori  estimators $\E_N^{\Ss,j}$ and corresponding $\text{EOC}_{\Ss,j},\, 0\le j\le 3$.} \label{tspace}
\end{center}
\end{table}
\begin{table}[!ht]
\begin{center}
\renewcommand{\arraystretch}{1.1}
\begin{tabular}{|c||cc|cc|cc|}\hline
$k^{-1}$ &  $ \E_N^{\TT,0}$ & $\text{EOC}_{\TT,0}$&$\E_N^{\TT,1}$ & $\text{EOC}_{\TT,1}$& $\E_N^{\TT,2}$& $\text{EOC}_{\TT,2}$ \\ \hline\hline
$ 252 $& $ 8.2786$E$-06 $& --          &$ 3.1156$E$-04 $&  --        & $ 1.6442$E$-05 $& -- \\
$ 464 $& $ 2.4426$E$-06 $& $2.0055$ &$ 9.2190$E$-05$ &$ 2.0007 $& $ 4.8667$E$-06 $& $2.0002$ \\
$ 715 $& $ 1.0288$E$-06 $& $2.0032$ &$ 3.8879$E$-05$ &$ 2.0003 $& $ 2.0526$E$-06 $& $2.0001$ \\
$1000$ &$ 5.2600$E$-07$ &$ 2.0021$ &$ 1.9891$E$-05$ &$ 2.0000$ &$ 1.0501$E$-06$ &$ 2.0000$ \\
$1314$ &$ 3.0466$E$-07$ &$ 2.0016$ &$ 1.1528$E$-05$ &$ 1.9993$ &$ 6.0850$E$-07$ &$ 2.0000$ \\
$1656$ &$ 1.9182$E$-07$ &$ 2.0012$ &$ 7.2678$E$-06$ &$ 1.9957$ &$ 3.8324$E$-07$ &$ 2.0000$ \\
$2023$ &$ 1.2854$E$-07$ &$ 2.0010$ &$ 4.8767$E$-06$ &$ 1.9943$ &$ 2.5686$E$-07$ &$ 2.0000$ \\
\hline
\end{tabular}\\[2ex]
\caption{Time a posteriori  estimators $\E_N^{\TT,j}$  and corresponding $\text{EOC}_{\TT,j},\, 0\le j\le 2$.} \label{ttime}
\end{center}
\end{table}

We are also interested in computing the \emph{effectivity index} which is the ratio between an a posteriori error estimator and the exact error. The effectivity index is a tool providing information on the quality of the estimator. To that end, we denote by $\mathrm{E_\text{exact}}:=\displaystyle\max_{0\le n\le N}\|u(t_n)-U^n\|$ the exact error, by $\E_N$ the sum of all time and space estimators and we compute the effectivity index $ei$ as $ei:=\E_N/\mathrm{E_\text{exact}}.$ The exact error $\mathrm{E_\text{exact}}$, the estimator $\E_N$ and the effectivity index $ei$ are presented in Table~\ref{ttteei}. We observe  that the computed effectivity index $ei$ stabilizes to a fixed value, indicating in this case that the sum of all time and space estimators $\E_N$ is about $26$ times larger than the actual error.
\begin{table}[!ht]
\begin{center}
\renewcommand{\arraystretch}{1.1}
\begin{tabular}{|c|c|c|}\hline
$\E_N$ & $\mathrm{E_\text{exact}} $ &  $ei$ \\ \hline\hline
$4.1688$E$-04 $& $1.6587$E$-05 $& $25.13$ \\
$1.2334$E$-04 $& $4.9118$E$-06 $& $25.11 $\\
$5.2012$E$-05 $& $2.0186$E$-06 $& $25.77$ \\
$2.6612$E$-05 $& $1.0202$E$-06 $& $26.09$ \\
$1.5418$E$-05 $& $5.8838$E$-07 $& $26.20$ \\
$9.7171$E$-06 $& $3.7056$E$-07 $& $26.22$ \\
$6.5176$E$-06 $& $2.4818$E$-07 $& $26.26$ \\
\hline
\end{tabular}\\[2ex]
\caption{A posteriori  estimator $\E_N$, exact error $\mathrm{E_\text{exact}}$ and effectivity index $ei$.} \label{ttteei}
\end{center}
\end{table}

Finally we would like to note that the data estimator $\E_N^{\mathrm{D}}$ is, for this example, much smaller than the time and space estimators, whilst the coarsening estimator $\E_N^{\mathrm{C}}$ is zero in the case of uniform partitions. This is the reason we don't include these estimators in $\E_N.$
\section*{Acknowledgments}
I.K. is grateful to  Prof.\ Charalambos Makridakis for suggesting the problem and for fruitful discussions.


\end{document}